\documentclass[smallextended,numbook,envcountsect,envcountsame]{svjour3}
\journalname{Archive for Rational Mechanics and Analysis}

\usepackage{amsmath}
\usepackage{amssymb}
\smartqed

\usepackage[pdftex]{graphicx}
\usepackage{subfig}
\usepackage{afterpage}

\usepackage[hyperindex=false, pagebackref=true, linkcolor=blue, citecolor=blue, colorlinks=true]{hyperref}

\hypersetup{
 pdfauthor = {Isaac Vikram Chenchiah, Anja Schl\"omerkemper},
 pdftitle = {Non-laminate Microstructures in Monoclinic-I Martensite}
}

\spnewtheorem{algorithm}[theorem]{Algorithm}{\bf}{\it}
\spnewtheorem{cons}[theorem]{Construction}{\it}{\rm}
\spnewtheorem{obs}[theorem]{Observation}{\bf}{\rm}

\usepackage{cancel} 
\usepackage{stmaryrd} 

\newcommand{\compatible}{\ensuremath{\interleave}}
\newcommand{\incompatible}{\ensuremath{\cancel{\interleave}}}

\newcommand{\eps}{\varepsilon}

\newcommand{\R}{\mathbb{R}}
\newcommand{\N}{\mathbb{N}}

\newcommand{\E}{\ensuremath{\mathcal{E}}}
\newcommand{\I}{\ensuremath{\mathcal{I}}}
\renewcommand{\L}{\ensuremath{\mathcal{L}}}
\newcommand{\T}{\ensuremath{\mathcal{T}}}

\newcommand{\CH}{\ensuremath{\mathcal{C}}}
\newcommand{\QH}{\ensuremath{\mathcal{Q}}}
\newcommand{\RH}{\ensuremath{\mathcal{R}}}
\newcommand{\LH}{\ensuremath{\mathcal{L}}}

\newcommand{\F}{\ensuremath{\mathcal{F}}}
\newcommand{\G}{\ensuremath{\mathcal{G}}}

\newcommand{\aff}{\ensuremath{\operatorname{aff}}}
\newcommand{\cof}{\ensuremath{\operatorname{cof}}}
\newcommand{\interior}{\ensuremath{\operatorname{int}}}
\newcommand{\rel}{\ensuremath{\operatorname{rel}}}
\newcommand{\sign}{\ensuremath{\operatorname{sign}}}
\newcommand{\Span}{\ensuremath{\operatorname{span}}}
\newcommand{\Tr}{\ensuremath{\operatorname{Tr}}}

\newcommand{\Sym}[2][]
 {\ensuremath{{\mathcal S}^{#2 \times #2}_{\text{\ensuremath{#1}}} }}

\begin{document}

\title{Non-laminate Microstructures in Monoclinic-I Martensite} 

\author{Isaac Vikram Chenchiah \and Anja Schl\"omerkemper}

\institute{
  Isaac Vikram Chenchiah \at School of Mathematics, University of Bristol, University Walk, Bristol BS8~1TW, UK.
  \email{Isaac.Chenchiah@bristol.ac.uk} 
  \and 
  Anja Schl\"omerkemper \at Chair of Mathematics in the Sciences, Institute for Mathematics, University of W\"urzburg, Emil-Fischer-Str.\ 40, 97074 W\"urzburg, Germany. \email{Anja.Schloemerkemper@mathematik.uni-wuerzburg.de} 
}

\date{}

\maketitle

\abstract{
We study the symmetrised rank-one convex hull of monoclinic-I martensite (a twelve-variant material) in the context of geometrically-linear elasticity. We construct sets of $T_3$s, which are (non-trivial) symmetrised rank-one convex hulls of three-tuples of pairwise incompatible strains. Moreover we construct a five-dimensional continuum of $T_3$s and show that its intersection with the boundary of the symmetrised rank-one convex hull is four-dimensional. We also show that there is another kind of monoclinic-I martensite with qualitatively different semi-convex hulls which, so far as we know, has not been experimentally observed.

Our strategy is to combine understanding of the algebraic structure of symmetrised rank-one convex cones with knowledge of the faceting structure of the convex polytope formed by the strains.
}


\section{Introduction}

Shape-memory alloys are materials that undergo a diffusionless solid-to-solid phase transformation due to change of temperature. They are capable of large macroscopic deformations and recover their original shape upon heating.  While such materials usually form a cubic lattice (austenite) above a critical temperature they develop microstructures at lower temperatures (martensite). In this article we are interested in the cubic-to-monoclinic-I phase transformation, which occurs, eg., in   NiTi, which is industrially one of the most important shape-memory alloys. In this case, there are twelve transformation strains, see Section~\ref{sec:strains} for details.

Of interest is the set of all strains that can be recovered upon heating. In the variational approach to martensite~\cite{Ball:1987}  this set is modelled by the quasiconvex hull (Definition~\ref{def:semi-convex}) of the transformation strains. Unfortunately the quasiconvex hull of a set is difficult to calculate. 

Bhattacharya and Kohn \cite{Bhattacharya:1997p99} consider various phase transformations (cubic to tetragonal, cubic to trigonal, cubic to orthorhombic and cubic to monoclinic) in the context of geometrically-linear elasticity and observe that except for cubic-to-monoclinic martensite the symmetrised quasiconvex hull coincides with the convex hull. For cubic-to-monoclinic martensite they show that the symmetrised quasiconvex hull is strictly smaller than the convex hull and present an inner bound for it.

We aim to find a better approximation of the symmetrised quasiconvex hull, again in the context of geometrically-linear elasticity. To this end, we are interested in the symmetrised lamination convex hull and the symmetrised rank-one convex hull of the transformation strains; these give inner bounds on the symmetrised quasiconvex hull, see Section~\ref{sec:semi-convexity}.  

Our analysis shows that there are points in the symmetrised rank-one convex hull of mono\-clinic-I martensite which are attained by \emph{non-laminate} microstructures. This suggests (see below) that the symmetrised rank-one convex hull is strictly larger than the lamination convex hull of the twelve transformation strains. Since the symmetrised rank-one convex hull is a subset of the symmetrised quasiconvex hull (Remark~\ref{rem:semi-convex}), the strains attained by these non-laminate microstructures belong to the set of recoverable strains.

Next we give more details on the strategy we use, our results and the organisation of the paper. Finally we will fix some notation.

\paragraph{Strategy.}
Our strategy is to exploit firstly, the algebraic structure of symmetrised rank-one convex cones, secondly, the faceting structure of the convex hull of a finite set, and thirdly, the interaction between the two. 

That rank-one convex cones are varieties has been exploited to develop algorithms to calculate semi-convex hulls~\cite{Kreiner:2003,Kreiner:2004}. The relevance of convex polytopes to semi-convex hulls (compare Lemma~\ref{lem:pairwise-compatible} with Theorem~\ref{thm:edge-compatible}) has been noticed in~\cite{Zhang:2006,Tang:2008}  but has not, to our knowledge, been exploited to determine semi-convex hulls. This paper represents a first attempt in this direction.

The central idea is the following: Given a finite set whose symmetrised rank-one convex hull we wish to compute we proceed as follows. First we compute the symmetrised rank-one convex hulls of all its one-dimensional facets, i.e., edges in the language of convex polytopes; this is trivial, cf.~\eqref{eq:1D} below. We use this, together with knowledge of the structure of symmetrised rank-one convex cones on two-dimensional affine subspaces (Section~\ref{sec:2D-cone}), to determine the symmetrised rank-one convex hulls of all its two-dimensional facets. We then repeat this for higher dimensions.

When the finite set we are interested in is the set of transformation strains of a material capable of a phase transformation from austenite to martensite, the set lies in the five-dimensional affine plane of strains with constant trace and thus the process above terminates when the symmetrised rank-one convex hull of the five-dimensional facet of the set (which is the convex hull of the set) is computed.

In this bootstrapping strategy the steps become progressively more difficult as the dimension increases. Indeed while we can completely implement the two-dimensional step (Section~\ref{sec:2D-cone} and~\cite{Chenchiah-Schloemerkemper-PRSL}) and have partial results for the three-dimensional step~\cite{Chenchiah-Schloemerkemper-JMPS} we have a reason to believe that the steps for dimension four and five are considerably more difficult than that for steps two and three: unlike in lower dimensions, the symmetrised rank-one convex cone in the higher dimensions is an algebraic surface of a polynomial which is necessarily irreduccible~\cite{Chenchiah-Schloemerkemper-PRSL}.

\paragraph{Results.}
We have three main results:

First we show that there are two kinds of monoclinic-I martensites which differ qualitatively with regard to the polytope-structure of their convex hulls. It follows that their semi-convex hulls are qualitatively different as well. Curiously all known monoclinic-I martensites belong to one of these kinds, which we name monoclinic-Ia (the other being monoclinic-Ib).

The question as to whether $T_3$s (which are non-trivial symmetrised rank-one convex hulls of 3-tuples of pairwise incompatible strains, which are attained by non-laminate microstructures, see Section~\ref{sec:T3}) can be formed from the twelve variants of Monoclinic-II martensite is raised in~\cite[p863]{Bhattacharya:1994p843}. There it is shown that this is possible when a certain lattice parameter is sufficiently small. Here we prove a stronger result: In Section~\ref{sec:T3} we present a simple test for $T_3$s (Lemma~\ref{lem:T3}) that shows that in fact $T_3$s can form for all (non-zero) values of the lattice parameter, and that the same is true for Monoclinic-I martensite as well. This is our second main result.

Our third result is a consequence of this: We show that for Monoclinic-Ia martensite, the symmetrised rank-one convex hull of the twelve transformation strains contains a five-dimensional continuum of points which are attained by non-laminate microstructures. Moreover the intersection of this continuum with the boundary of the convex hull is four-dimensional. This suggests that the symmetrised rank-one convex hull of the transformation strains is strictly larger than the symmetrised lamination convex hull. This would then imply that the symmetrised quasiconvex envelope of the energy density of this material is different from the symmetrised lamination convex envelope.  It is well known that lamination convex envelopes can differ from rank-one convex envelopes, cf.\ \cite[Sect.\ 4]{Dacorogna:2007}, and (in dimensions larger than two) that rank-one convex envelopes can differ from quasiconvex envelopes~\cite{Sverak:1992}. However, in the context of materials science, all quasiconvex and rank-one convex envelopes that have been evaluated so far have in fact  coincided with lamination convex envelopes; Monoclinic-Ia martensite is the first \emph{material} for which we now have a strong indication that they differ.

\paragraph{Organisation of the paper.}
In Section~\ref{sec:sroc} we refresh the reader's memory of some basic facts and results about strain compatibility, convex sets (in particular, convex polytopes) and semi-convex functions and sets.

In Section~\ref{sec:2D-cone} we study the structure of symmetrised rank-one convex cones on two-dimensional affine subspaces of $\Sym[c]{3}$ (Lemma~\ref{lem:2D-compatible_directions}), see below for notation. The results presented here enable the computation of the symmetrised rank-one convex hull of any finite set in two-dimensional affine subspaces of $\Sym[c]{3}$ and the characterisation of those compact sets in these spaces that possess non-trivial symmetrised rank-one convex hulls; we present some results in Section~\ref{sec:T3} but postpone a more extensive discussion to~\cite{Chenchiah-Schloemerkemper-PRSL}. We extend these results to higher dimensions in~\cite{Chenchiah-Schloemerkemper-JMPS} but Lemma~\ref{lem:T3-3D} and Section~\ref{sec:nonlaminates} provide a glimpse of the utility of the results in Section~\ref{sec:T3} even in higher dimensions.

We then turn from abstract results to the specific class of materials of interest to us, monoclinic-I martensite. After some preliminary observations in Section~\ref{sec:strains} on the compatibility and symmetry relations between the twelve transformation strains of materials in this class we determine, in Section~\ref{sec:polytope}, the facets of the convex hull of the twelve transformation strains of monoclinic-I martensite (Observations~\ref{obs:e<d}, \ref{obs:e>d} and \ref{obs:e=d}). This leads to the discovery that there are in fact two kinds of monoclinic-I martensitic materials.

With this foundation behind us, in Section~\ref{sec:nonlaminates} we investigate the (theoretical) possibility of non-laminate zero-energy microstructures occurring in these materials. We construct an open set in the symmetrised rank-one convex hull of the transformation strains for which $T_3$-microstructures are optimal (Construction~\ref{cons:5D-set-T3s}). We then deduce that in monoclinic-Ia martensite this set intersects the boundary of the convex hull, and thus the boundary of the symmetrised rank-one convex hull.

(In Sections~\ref{sec:polytope} and~\ref{sec:nonlaminates} we use Mathematica to simplify computations but these computations are non-numerical. The Mathematica code that we have used together with explanatory notes can be found in the electronic supplementary material accompanying this article~\cite{Chenchiah-Schloemerkemper1ESM}.)

We conclude with Section~\ref{sec:conclusions} with some questions raised by the preceding two sections. One of these is whether monoclinic-Ib martensite might have a larger set of recoverable strains (i.e., quasiconvex hull) than monoclinic-Ia martensite (modulo appropriate normalisation of the lattice parameters). This naturally also leads to the question as to whether a material that lies at the boundary of monoclinic-Ia and monoclinc-Ib martensite might demonstrate the best behaviour of all.

\paragraph{Notation.}
In the geometrically linear theory of elasticity, the strains (pointwise) belong to the space of real symmetric  $3\times 3$ matrices denoted by $\Sym{3}$. We phrase some of our results for real symmetric $d\times d$ matrices and then use the symbol $\Sym{d}$. The space of real symmetric $d \times d$ matrices whose trace is an arbitrary (but fixed) constant $c$ is denoted by $\Sym[c]{d}$.

We introduce an inner product $\langle \cdot, \cdot \rangle$ on $\Sym{d}$ by $\Sym{d} \ni A,B \mapsto \langle A, B \rangle = \Tr(AB)$; the norm induced by this inner product is $\| \cdot \|$. 

For $e_1,e_2 \in \Sym{d}$ we set
\begin{align*}
(e_1,e_2)
  &:= \{ \lambda e_1 + (1-\lambda) e_2\ |\ \lambda \in (0,1) \}, \\
[e_1,e_2]
  &:= \{ \lambda e_1 + (1-\lambda) e_2\ |\ \lambda \in [0,1] \}.
\end{align*}
By a \emph{direction} in $\Sym{d}$ we mean a one-dimensional affine subspace of $\Sym{d}$. We denote the affine span of $S \subset \Sym{d}$ by $\aff \Span (S)$,
the relative boundary of $S \subset \Sym{d}$ by $\rel \partial S$ and its relative interior by $\rel \interior {S}$.

\section{Preliminaries}
\label{sec:sroc}

\subsection{Strain compatibility}
\label{sec:compatibility}

\begin{definition}[Strain compatibility] \label{def:compatibility}
Let $e_1, e_2 \in \Sym{3}$ and $S_1, S_2 \subset \Sym{3}$.
\begin{enumerate}
\item $e_1$ is compatible with $e_2$ (or $e_1$ and $e_2$ are compatible), $e_1 \compatible e_2$, if there exist $a,b \in \R^3$ such that $e_1 - e_2 = \frac12 (a \otimes b + b \otimes a)$, where $\otimes$ denotes the dyadic product. 
\item $e_1$ is compatible with 0 (or compatible for short) if $e_1 \compatible 0$.
\item $e_1$ is incompatible with $e_2$ (or $e_1$ and $e_2$ are incompatible), $e_1 \incompatible e_2$, if $e_1$ and $e_2$ are not compatible.
\item $e_1$ is incompatible with 0 (or incompatible for short) if $e_1 \incompatible 0$.
\item $S_1$ is totally compatible with $S_2$ (or $S_1$ and $S_2$ are totally compatible), $S_1 \compatible S_2$,  if for all $e_1 \in S_1$ and for all $e_2 \in S_2$, $e_1 \compatible e_2$. \label{it:local1}  
\item $S_1$ is compatible if $S_1 \compatible S_1$, i.e., if for all $e_1,e_2 \in S_1$, $e_1 \compatible e_2$.  \label{it:local9}
\item $S_1$ is totally incompatible with $S_2$ (or $S_1$ and $S_2$ are totally incompatible),
  $S_1 \incompatible S_2$, if for all $e_1 \in S_1$ and for
  all $e_2 \in S_2$, $e_1 \incompatible e_2$. \label{it:local2} 
\end{enumerate}
\end{definition}

We observe that $[e_1,e_2]$ is compatible if and only if $e_1 \compatible e_2$. An alternative term for compatibility is ``symmetrised rank-one connectedness''.

In the figures compatible lines are represented by solid lines and incompatible lines by dashed lines, cf.\ eg.\ Figure~\ref{fig:T3}. The following lemma follows immediately from~\cite[Lemma 4.1]{Kohn:1991p193}. 

\begin{lemma}[Compatibility in ${\Sym[c]{3}}$]
\label{lem:SymTr3-compatibility}
Let $e_1, e_2 \in \Sym[c]{3}$. Then
$e_1 \compatible e_2$ iff $\det(e_1 - e_2) = 0$.
\end{lemma}

\begin{definition}[Compatible cone]
Let $S \subset \Sym[c]{3}$ and $x \in S$. The compatible cone in $S$ at $x$ is the set
\begin{equation*}
\Lambda_{S,x}
 := \{ y \in S\ |\ y \compatible x \}.
\end{equation*}
When $0 \in S$ we set $\Lambda_S := \Lambda_{S,0}$.
\end{definition}

Note that the compatible cone $\Lambda_{S,x}$ is an affine cone with vertex $x$ in the linear algebraic sense, i.e., it is closed under multiplication by positive reals. However, the compatible cone is not geometrically a cone.

\subsection{Convex sets}
\label{sec:convexity}

We recall some elementary definitions and results from convex analysis. For more details see, eg.\ \cite{Rockafellar:1996}.

Let $E$ be a subset of a vector space. We denote the convex hull of $E$ by $\CH(E)$.

\begin{definition}[Extreme subsets of convex sets]
Let $S\subset \R^d$ be convex. Then $S' \subseteq S$ is an {\em extreme subset} of $S$ if $S'$ is convex and satisfies: If $x,y\in S$ and $\exists \lambda \in (0,1)$ such that $\lambda x + (1-\lambda)y\in S'$, then $x,y\in S'$. 
\end{definition}

Of special interest to us are \emph{convex polytopes} which are convex hulls of finite sets. Definition~\ref{def:polytope} and Remark~\ref{rem:polytope} suffice for us. For an introduction to convex polytopes we refer the reader to, eg.\ \cite{Barvinok:2002,Brondsted:1982,Ewald:1996,Grunbaum:2003,Gruber:2007,Ziegler:1994}.

\begin{definition}[Vertices, edges and facets of convex polytopes]
\label{def:polytope}
The vertices of a convex polytope are its extreme points, its edges are its one-dimensional extreme subsets and its facets are its extreme subsets with co-dimension one.
\end{definition}

Let $E$ be a finite set. We denote the set of $n$-dimensional extreme subsets of $\CH(E)$ by $\F_n(E)$. Thus the set of vertices of $\CH(E)$ is $\F_0(E)$, the set of its edges is $\F_1(E)$, and the set of facets is $\F_{\dim(\CH(E))-1}(E)$ and $\F_{\dim{\CH(E)}}(E) = \{ \CH(E) \}$.

\begin{remark}
\label{rem:polytope}
 For $S \in \F_{n}(E)$, $n=1,\dots, \dim(\CH(E))$, the relative boundary $\rel \partial S$ is a union of elements of $\F_{n-1}(E)$. 
\end{remark}

\subsection{Semi-convex functions and sets}
\label{sec:semi-convexity}

Next we recall some elementary definitions and results about semi-convex functions and sets. For more details see, eg.\ \cite{Muller:1999,Dacorogna:2007}. Since we work in the context of $S^{d \times d}$ (as opposed to $R^{d \times d}$) we have appended the qualifier ``symmetrised'' to the various notions of semi-convexity. The qualifier ``symmetric'' is also used in the literature.

The term ``symmetrised rank-one convex'' (Defintion~\ref{def:semi-convex} below) might be misleading in that the matrices involved need not be of rank one but rather are the symmetric parts of rank-one matrices. An alternative name could be ``wave-cone convex'', because the symmetric tensor products form the wave cone of a second-order linear differential operator whose kernel are the symmetrised gradients~\cite{Tartar:1979,Murat:1981}.

\begin{definition}[Semi-convex functions]
\par \indent
\begin{enumerate}
\item $f \colon S^{d \times d} \to \R$ is symmetrised rank-one convex if for all $\lambda \in [0,1]$
\begin{equation*}
f(\lambda e_1 + (1-\lambda) e_2)
 \leqslant \lambda f(e_1) + (1-\lambda) f(e_2) \quad \forall e_1, e_2 \in S^{d \times d} \text{ with } e_1 \compatible e_2.
\end{equation*}
\item A locally-bounded Borel function $f \colon S^{d \times d} \to \R$ is symmetrised quasiconvex if for an open and bounded set $U\subset \R^d$ with $|\partial U|=0$ one has
\begin{equation*}
f(e)
 \leqslant \frac{1}{|U|} \int_U f(e + D_s\varphi) \,dx
 \quad \forall \varphi \in W^{1,\infty}_0(U,\R^d),
\end{equation*}
whenever the integral on the right hand side exists, where $D_s \varphi$ is the symmetrised gradient of $\varphi$.
\end{enumerate}
\end{definition}

\begin{definition}[Semi-convex sets]
\label{def:semi-convex}
Let $E$ be a compact set in $S^{d \times d}$.
\begin{enumerate}
\item The symmetrised quasiconvex hull of $E$ is defined as
\begin{equation*}
\QH(E)
 := \{ e \in S^{d \times d} \, | \, f(e) \leqslant \sup_{e' \in E}
      f(e') \quad \forall f \colon S^{d \times d} \to \R \text{ quasiconvex} \}.
\end{equation*}
\item The symmetrised rank-one convex hull of $E$ is defined as
\begin{multline*}
\RH(E)
 := \{ e\in S^{d \times d} \, | \, f(e) \leqslant \sup_{e' \in E} f(e') \\ 
      \forall f \colon S^{d \times d} \to \R
       \text{ symmetrised rank-one convex} \}.
\end{multline*}
\item The symmetrised lamination convex hull of $E$, $\LH(E)$, is the
smallest set $\widetilde{E} \supseteq E$ such that $e_1, e_2 \in \widetilde{E}$
and $e_1 \compatible e_2$ implies that $[e_1, e_2] \subset \widetilde{E}$.
\end{enumerate}
\end{definition}

We note that symmetrised lamination convex hulls can be constructed as follows: For $n \in \N_0$, let
\begin{align*}
\LH_0(E) &:= E, \\
\LH_{n+1}(E) &:= \left\{ e\ |\ \exists e_1, e_2 \in \LH_n(E),\ e_1 \compatible e_2,\ e \in [e_1,e_2] \right\}.
\end{align*}
Then,
\begin{equation*}
\LH(E) = \bigcup_{n \in \N_0} \LH_n(E).
\end{equation*}

Symmetrised rank-one convex hulls also have an alternate characterisation (cf.\ eg.\ \cite[Lemma 2.5]{Kreiner:2003}):
\begin{multline} \label{eq:sroch}
\RH(E)
 := \{ e\in S^{d \times d} \, | \, f(e) = 0 
      \forall f \colon S^{d \times d} \to [0,\infty) \\
       \text{ symmetrised rank-one convex with } f(E)=\{0\} \}.
\end{multline}

We shall repeatedly use Remark~\ref{rem:semi-convex} and Lemma~\ref{lem:pairwise-compatible} without explicitly citing them:

\begin{remark}
\label{rem:semi-convex}
$\LH(E) \subseteq \RH(E) \subseteq \QH(E) \subseteq \CH(E)$. 
\end{remark}

\begin{lemma}[{\cite[Section~3.4.1, p.~231]{Bhattacharya:1993p205}}]
\label{lem:pairwise-compatible}
Let $E \subset \Sym{d}$ be finite and compatible.   
Then $\LH(E) = \CH(E)$.
\end{lemma}

In particular, 
\begin{alignat}{2}
&\forall x,y \in \Sym[c]{3}, \quad
 &\RH(\{x,y\}) 
  &= \begin{cases}
      [x,y] & \text{if } x  \compatible y, \\
      \{x,y\} & \text{else}.
      \end{cases} \label{eq:1D}
\end{alignat}

Lemma~\ref{lem:pairwise-compatible} however is too weak for our purposes; we present a sharp version~\cite[p.38]{Zhang:2006}\cite[p.1266]{Tang:2008}:

\begin{theorem}
\label{thm:edge-compatible}
Let $E \subset \Sym{d}$ be finite. Then $\LH(E) = \CH(E)$ if and only if every edge of $\CH(E)$ is compatible, i.e., if and only if
\begin{equation*}
\forall e_1, e_2 \in E, \qquad \ e_1 \incompatible e_2 \implies [e_1,e_2] \notin \F_1(E).
\end{equation*}
\end{theorem}

\begin{proof}
\emph{Necessity}.
Let $[e_1,e_2] \in \F_1(E)$ with $e_1 \incompatible e_2$. We show that $(e_1,e_2) \not\subset \LH(E)$:

Let $e \in [e_1,e_2]$ and $e \in \LH(E)$. Then $\exists n \in \N$ finite,  $f_1, f_2 \in \LH_n(E)$ such that $e \in [f_1,f_2]$. However by extremality of $[e_1,e_2]$, $f_1,f_2 \in [e_1,e_2]$. Applying the same argument to $f_1$ and $f_2$ we conclude that in fact $e \in \LH(\{e_1,e_2\})$ but from~\eqref{eq:1D}, $\LH(\{e_1,e_2\}) = \{e_1,e_2\}$. Thus $e \in [e_1,e_2]$ and $e \in \LH(E) \implies e \in \{e_1,e_2\}$.

\emph{Sufficiency}.
Let every edge of $\CH(E)$ be compatible. We show by induction that
\begin{equation}
S \in \F_n(E) \implies S \subseteq \LH_n(E). \tag{$P_n$}
\end{equation}
for $n=0,1,\dots, \dim(\CH(E))$. 
Since ($P_{\dim(\CH(E))}$) is the statement that $\CH(E) \subseteq \LH(E)$, the result follows from Remark~\ref{rem:semi-convex}.

($P_0$) is trivially true since $e \in \F_0$ implies $e \in E$ and thus $e \in \LH_0(E)$; ($P_1$) is true since every edge of $\CH(E)$ is compatible by assumption. Now let ($P_n$) be true for some fixed $n=1,\dots, \dim(\CH(E))-1$. We show that ($P_{n+1}$) is true:

Let $S \in \F_{n+1}(E)$ and $e \in S$. If $e \in \rel \partial S$, then $e$ is contained in an element of $\F_{n}(E)$ (Remark~\ref{rem:polytope}) and thus, by the inductive hypothesis, $e \in \LH_n(E) \subset \LH_{n+1}(E)$. We consider the case $e \in \rel \interior S$:

Pick $v \in \F_1(E)\cap \rel\partial S$. We view $v$ as a vector in $\Sym{d}$. Since $S$ is bounded, the inclusion 
\begin{equation*}
 e + \lambda v \in \rel \partial S, \quad \lambda \in \R,
\end{equation*}
has two solutions $\lambda_- < 0$ and $\lambda_+ > 0$. Note that $(e + \lambda_+ v)-(e + \lambda_- v)$ is parallel to $v$, which is compatible (by assumption). Thus $e + \lambda_\pm v \in \LH_n(E)$ (by ($P_n$)), $e + \lambda_- v \compatible e + \lambda_+ v$ and
\begin{equation*}
e
 = \frac{\lambda_+}{\lambda_+ - \lambda_-} (e + \lambda_- v) + \frac{-\lambda_-}{\lambda_+ - \lambda_-} (e + \lambda_+ v).
\end{equation*}
It follows that $e \in \LH_{n+1}(E)$ and thus $S \subseteq \LH_{n+1}(E)$.

\end{proof}

\section{The compatible cone in two-dimensional affine subspaces of $\Sym[c]{3}$}
\label{sec:2D-cone}

Our first task is to characterise, both geometrically (Lemma~\ref{lem:2D-compatible_directions}) and algebraically (Remark~\ref{rem:2D-compatible_directions}), the symmetrised rank-one convex cone in two-dimensional affine subspaces of $\Sym[c]{3}$.

\begin{lemma}
\label{lem:2D-compatible_directions}
Let $S$ be a two-dimensional subspace of $\Sym[c]{3}$. Then either 
\begin{enumerate}
 \item $S$ contains precisely one, two or three compatible directions, or
 \item $S$ is compatible.
\end{enumerate}
\end{lemma}

\begin{proof}
Let $\{e_1,e_2\}$ be a basis for $S$. By Lemma~\ref{lem:SymTr3-compatibility}, an arbitrary non-zero element of $S$, $xe_1 + ye_2$, $(x,y) \neq 0$, is compatible iff
\begin{equation*}
\det(xe_1 + ye_2)
 = 0.
\end{equation*} 
This is equivalent to
\begin{equation}
\label{eq:local6}
x^3 \det(e_1) + x^2 y \langle \cof(e_1), e_2 \rangle + x y^2 \langle e_1, \cof(e_2) \rangle + y^3 \det(e_2)
 = 0,
\end{equation}
where $\cof(e)$ is the cofactor of $e$:
\begin{multline*}
\cof \begin{pmatrix}
       e_{11} & e_{12} & e_{13} \\
       e_{21} & e_{22} & e_{23} \\
       e_{31} & e_{32} & e_{33}
      \end{pmatrix} \\
 = \begin{pmatrix}
       e_{22} e_{33} - e_{32} e_{23} & -\left( e_{21} e_{33} - e_{31} e_{23} \right) & e_{21} e_{32} - e_{31} e_{22} \\
       -\left( e_{12} e_{33} - e_{32} e_{13} \right) & e_{11} e_{33} - e_{31} e_{13} & -\left( e_{11} e_{32} - e_{31} e_{12} \right) \\
       e_{12} e_{23} - e_{22} e_{13} & -\left( e_{11} e_{23} - e_{21} e_{13} \right) & e_{11} e_{22} - e_{21} e_{12}
      \end{pmatrix}^T.
\end{multline*}
If the polynomial in~\eqref{eq:local6} is the zero-polynomial (i.e., $\det(e_1)=0, \langle \cof(e_1), e_2 \rangle=0$ etc.) then $S$ is compatible. Otherwise:
 
If either $\det(e_1) = 0$ or $\det(e_2) = 0$ then \eqref{eq:local6} has one or two solutions $(x,y) \neq 0$ and thus $S$ contains one or two compatible directions. Assume on the contrary that both $\det(e_1)$ and $\det(e_2) \neq 0$. Since $(x,y) \neq 0$, dividing~\eqref{eq:local6} by either $x^3$ or $y^3$ we obtain a polynomial (in either $\frac{x}{y}$ or $\frac{y}{x}$) which, being cubic, has one, two or three distinct real roots. Thus $S$ contains one, two or three compatible directions. 
\end{proof}

\begin{corollary}
Let $S$ be a subspace of $\Sym[c]{3}$ with $\dim(S) > 1$. Then $S$ contains at least one compatible direction.
\end{corollary}

The following examples show that each of the possibilities referred to in Lemma~\ref{lem:2D-compatible_directions} can occur:

\begin{example}[One compatible direction] \par \indent
The determinant of these matrices is proportional to $x(x^2+y^2)$ so the compatible direction is $x=0$:
\begin{equation*}
\left\{ \left. \begin{pmatrix} x & 0 & 0 \\ 0 & x & y \\ 0 & y & -2x \end{pmatrix} \right| x,y \in \R \right\}.
\end{equation*}
The determinant of these matrices is proportional to $x^3$ so the compatible direction is $x=0$:
\begin{equation*}
\left\{ \left. \begin{pmatrix} y & x & x \\ x & -y & x \\ x & x & 0 \end{pmatrix} \right| x,y \in \R \right\}.
\end{equation*} 
\end{example}

\begin{example}[Two compatible directions]
The determinant of these matrices is proportional to $xy^2$ so the two compatible directions are $x=0$ and $y=0$:
\begin{equation*}
\left\{ \left. \begin{pmatrix} x & 0 & 0 \\ 0 & -x & y \\ 0 & y & 0 \end{pmatrix} \right| x,y \in \R \right\}.
\end{equation*}
\end{example}

\begin{example}[Three compatible directions] \par \indent
The determinant of these matrices is proportional to $xy(x+y)$ so the three compatible directions are $x=0$, $y=0$ and $x=-y$:
\begin{equation*}
\left\{ \left. \begin{pmatrix} x & 0 & 0 \\ 0 & y & 0 \\ 0 & 0 & -x-y \end{pmatrix} \right| x,y \in \R \right\}. 
\end{equation*}
The determinant of these matrices is proportional to $x(x^2-y^2)$ so the three compatible directions are $x=0$, $x=y$ and $x=-y$:
\begin{equation*}
\left\{ \left. \begin{pmatrix} -2x & 0 & 0 \\ 0 & x & y \\ 0 & y & x \end{pmatrix} \right| x,y \in \R \right\}. 
\end{equation*}
\end{example}

\begin{example}[Compatible (two-dimensional) plane]
\begin{gather*}
\left\{ \left. \begin{pmatrix} 0 & 0 & 0 \\ 0 & x & y \\ 0 & y & -x \end{pmatrix} \right| x,y \in \R \right\}, \quad
\left\{ \left. \begin{pmatrix} 0 & x & y \\ x & 0 & 0 \\ y & 0 & 0 \end{pmatrix} \right| x,y \in \R \right\}.
\end{gather*}
\end{example}

We note that Lemma~\ref{lem:2D-compatible_directions} follows also from the following characterisation of real homogeneous cubic polynomials in two variables:

\begin{remark} 
\label{rem:2D-compatible_directions}
A homogeneous cubic polynomial on $\R^2$ can, by an appropriate choice of basis, be written in precisely one of the five following forms cf.\ eg.~\cite[Sec.23, p.263-6 ]{Gurevich:1964},\cite{Weinberg:1988p655} or~\cite[p.28]{Olver:1999}:
\begin{enumerate}
\item $x(x^2+y^2)$,
\item $x^3$,
\item $xy^2$,
\item $xy(x+y)$ or, equivalently, $x(x^2-y^2)$,
\item $0$.
\end{enumerate}
The examples above illustrate these possibilities.
\end{remark}

Lemma~\ref{lem:2D-compatible_directions} shows that from the perspective of symmetrised rank-one convexity there are four kinds of  two-dimensional affine subspaces of $\Sym[c]{3}$, namely those for which the compatible cone is
\begin{enumerate}
\item a line, \label{it:2D-1-line}
\item the union of two (distinct) lines, \label{it:2D-2-lines}
\item the union of three (distinct) lines, and \label{it:2D-3-lines}
\item the subspace itself. \label{it:2D-plane}
\end{enumerate}

The next step is to investigate compatible hulls in these subspaces. Cases~\eqref{it:2D-1-line} and~\eqref{it:2D-plane} are the simplest: In Case~\eqref{it:2D-1-line} compatible hulls are obtained by convexifying in the compatible direction, in Case~\eqref{it:2D-plane} compatible hulls are identical to convex hulls.

Cases~\eqref{it:2D-2-lines} and~\eqref{it:2D-3-lines} are reminiscent of separate convexity in $\R^2$ \cite{Tartar:1993p191,Kreiner:2003}: Symmetrised rank-one convexity is \emph{geometrically} identical to separate convexity in $\R^2$ in Case~\eqref{it:2D-2-lines}, and is similar to separate convexity in $\R^2$ in Case~\eqref{it:2D-3-lines}. As an aside, the inclusion-minimal configurations~\cite{Szekelyhidi:2005p253} in these situations are either $T_3s$ or $T_4s$ as might be expected. Only $T_3$s are relevant to our immediate purposes so we discuss them in Section~\ref{sec:T3} and leave the rest for \cite{Chenchiah-Schloemerkemper-PRSL}.

\section{$T_3$s}
\label{sec:T3}
 
$T3$s occur when there are precisely three directions in the compatible cone. Our definition is equivalent/identical to earlier definitions in the literature such as \cite{Bhattacharya:1994p843}[p.~855, (3.9) on p.~862 and Fig. 3.1b on p.~856].

\begin{definition}[$T_3$]  \label{def:T3-1}
Three points $e_1, e_2, e_3 \in \Sym[c]{3}$ form a $T_3$ if
\begin{enumerate}
\item They are pairwise incompatible, and \label{it:local3}
\item There exist $e_{1,1} \in (e_2,e_3), e_{2,2} \in (e_3,e_1), e_{3,3} \in (e_1,e_2)$ such that $e_{i,i} \compatible e_i$, $i=1,2,3$. \label{it:local4}
\end{enumerate}
\end{definition}

A schematic representation of a $T_3$ is shown in Figure~\ref{fig:T3}.

\begin{remark} \label{rem:T3}
Note that $\Span\{e_1 - e_{1,1}\}$, $\Span\{e_2 - e_{2,2}\}$ and $\Span\{e_3 - e_{3,3}\}$ are distinct compatible directions. By Lemma~\ref{lem:2D-compatible_directions} there are no others. It follows that $e_{1,1}$, $e_{2,2}$ and $e_{3,3}$ are unique. \label{it:local5}
\end{remark}

\begin{figure}
\begin{center} 
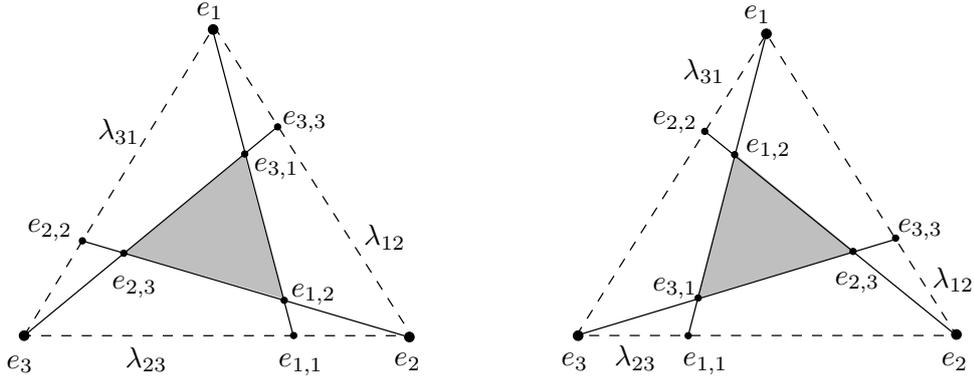
\caption{Schematic representations of points forming $T_3$s.} 
\label{fig:T3}
\end{center}
\end{figure}

Next we show that it is easy to check whether three points form a $T_3$:

\begin{lemma}
\label{lem:T3}
Three points $e_1, e_2, e_3 \in \Sym[0]{3}$ form a $T_3$ if and only if
\begin{equation}
\label{eq:local7}
\sign \det(e_1 - e_2) = \sign \det(e_2 - e_3) = \sign \det(e_3 - e_1) \neq 0.
\end{equation}
\end{lemma}

\begin{proof}
Let  $e_1, e_2, e_3 \in \Sym[0]{3}$ satisfy~\eqref{eq:local7}. From Lemma~\ref{lem:SymTr3-compatibility} it follows immediately  that Definition~\ref{def:T3-1}\eqref{it:local3} is satisfied. Consider the polynomial $[0,1] \ni \lambda \mapsto \det( (\lambda e_1 + (1-\lambda) e_2 ) - e_3 )$. By assumption, this polynomial has opposite signs at $0$ and $1$. In conjunction with Lemma~\ref{lem:SymTr3-compatibility} this implies:
\begin{subequations} \label{eq:T3-scaffold}
\begin{alignat}{2} \label{eq:T3-scaffolda}
\exists \lambda_{12} \in (0,1), \quad
 &e_{3,3} &:= (\lambda_{12} e_1 + (1-\lambda_{12}) e_2 ) &\compatible e_3.
\end{alignat}
Similarly,
\begin{alignat}{2} \label{eq:T3-scaffoldb}
\exists \lambda_{23} \in (0,1), \quad
 &e_{1,1} &:= (\lambda_{23} e_2 + (1-\lambda_{23}) e_3 ) &\compatible e_1, \\
\exists \lambda_{31} \in (0,1), \quad
 &e_{2,2} &:= (\lambda_{31} e_3 + (1-\lambda_{31}) e_1 ) &\compatible e_2. \label{eq:T3-scaffoldc}
\end{alignat}
\end{subequations}
This shows that Definition~\ref{def:T3-1}\eqref{it:local4} is satisfied. Thus $\{e_1, e_2, e_3\}$ forms a $T_3$.

Conversely, assume that $\{e_1, e_2, e_3\} \subset \Sym[0]{3}$
form a $T_3$. Then, from Definition~\ref{def:T3-1}\eqref{it:local3}, $\sign \det(e_1 - e_2), \sign \det(e_2 - e_3), \sign \det(e_3 - e_1) \neq 0$. Suppose it is \emph{not} the case that $\sign \det(e_1 - e_2) = \sign \det(e_2 - e_3) = \sign \det(e_3 - e_1)$. Then, relabelling the points if necessary, we have
\begin{equation}
\label{eq:local8}
\det(e_1 - e_3), \det(e_2 - e_3)
 > 0.
\end{equation}
Now consider the cubic polynomial $\R \ni \lambda \mapsto \det( (\lambda e_1 + (1-\lambda) e_2 ) - e_3 )$. By assumption there are three compatible directions passing through $e_3$, which are parallel to $[e_1,e_{1,1}]$, $[e_2,e_{2,2}]$ and $[e_3,e_{3,3}]$; see Figure~\ref{fig:T3}. Thus the cubic polynomical has three distinct real roots, one in each of $(-\infty,0)$, $(0,1)$ and $(1,\infty)$. Thus each root is simple; it follows that the polynomial changes sign around each root. 

On the other hand from~\eqref{eq:local8} the polynomial is positive at $0$ and $1$. It follows that it has one further root in $(0,1)$, which is a contradiction. We conclude that $\sign \det(e_1 - e_2) = \sign \det(e_2 - e_3) = \sign \det(e_3 - e_1)$.
\end{proof}

\begin{definition}[Vertices of a $T_3$]
Let $e_1, e_2, e_3 \in \Sym[c]{3}$ form a $T_3$ as in Definition~\ref{def:T3-1}. Then $e_1, e_2, e_3$ are referred to as the vertices of this $T_3$.
\end{definition}

When $e_1, e_2, e_3 \in \Sym[c]{3}$ form a $T_3$ then their symmetrised rank-one convex hull $\RH(\{e_1, e_2, e_3\})$ is (also) called a $T_3$. The context will make clear whether ``$T_3$'' refers to the set of three strains $\{e_1, e_2, e_3\}$ or to their symmetrised rank-one convex hull $\RH(\{e_1, e_2, e_3\})$.

It is known (cf., e.g., \cite[\S3]{Bhattacharya:1994p843}) that
\begin{equation} \label{eq:T3-hull}
\RH(\{e_1, e_2, e_3\})
 \supseteq [e_1,e_{1,2}] \cup [e_2,e_{2,3}] \cup [e_3,e_{3,1}] \cup \CH(\{e_{1,2}, e_{2,3}, e_{3,1}\}),
\end{equation}
where $e_{1,2}$, $e_{2,3}$, $e_{3,1}$ are the nodes of the $T_3$ (Definition~\ref{def:T3-nodes} below), see Figure~\ref{fig:T3-hull}. However for the convenience of the reader we provide a proof of this in Proposition~\ref{prop:T3-hull} below.

\begin{figure}
\begin{center} 
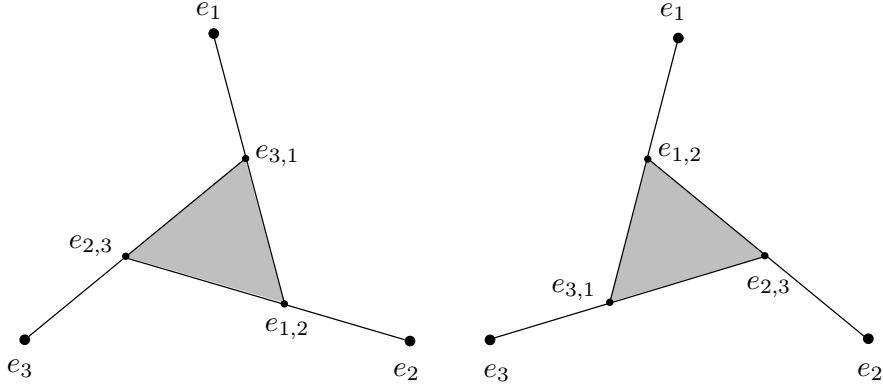
\caption{Schematic representations of $T_3$s, i.e. the rank-one convex hull of three strains forming a $T_3$.}
\label{fig:T3-hull}
\end{center}
\end{figure}

\begin{definition}[Nodes of a $T_3$] \label{def:T3-nodes}
Let $e_1, e_2, e_3 \in \Sym[c]{3}$ form a $T_3$. Let $e_{1,1}$, $e_{2,2}$, $e_{3,3}$ be as defined in~\eqref{eq:T3-scaffolda}--\eqref{eq:T3-scaffoldc}. We define the nodes of the $T_3$ to be the three points
\begin{equation*}
 e_{i,j} := [e_i,e_{i,i}] \cap [e_j,e_{j,j}], \quad i \neq j,\ i,j=1,2,3
\end{equation*}
 (see Figure~\ref{fig:T3}).
\end{definition}

From Definition~\ref{def:T3-1}\eqref{it:local4} the nodes of a $T_3$ are pair-wise compatible and $e_{i,j}$ is compatible with $e_i$ and $e_j$.

Later we will encounter symmetric $T_3$s and similar $T_3$s as defined below:

\begin{definition}[Symmetric $T_3$s] \label{def:T3-symmetry}
Let $e_1, e_2, e_3 \in \Sym[c]{3}$ form a $T_3$. Let $\lambda_{12}$, $\lambda_{23}$, $\lambda_{31}$ be as defined in~\eqref{eq:T3-scaffolda}--\eqref{eq:T3-scaffoldc}. Then the $T_3$ is symmetric if $\lambda_{12} = \lambda_{23} = \lambda_{31}$.
\end{definition}

The following remark is elementary but we explicitly state it since it arises frequently in applications (cf.\ Section~\ref{sec:nonlaminates}):

\begin{remark} \label{rem:T3-symmetry}
Let $R \in SO(3)$ such that $R^3$ is the identity. Let $e_1 \in \Sym[c]{3}$, $e_2 := R^T e_1 R$ and $e_3 := R^T e_2 R$. Then $\det(e_1-e_2) = \det(e_2-e_3) = \det(e_3-e_1)$ and if this is non-zero then $e_1$, $e_2$, $e_3$ form a symmetric $T_3$.
\end{remark}

\begin{definition}[Similar $T_3$s] \label{def:T3-similarity}
Let $e_1, e_2, e_3 \in \Sym[c]{3}$ and $e'_1, e'_2, e'_3 \in \Sym[c]{3}$ form $T_3$s. Let $\lambda_{ij}$, $\lambda'_{ij}$, $i \neq j$, $i,j=1,2,3$ be defined as in~\eqref{eq:T3-scaffold}.
\begin{enumerate}
\item We say that these $T_3$s are similar if, for some permutation $\sigma \colon \{1,2,3\} \to \{1,2,3\}$, $\lambda'_{12} = \lambda_{\sigma(1)\sigma(2)}$, $\lambda'_{23} = \lambda_{\sigma(2)\sigma(3)}$ and $\lambda'_{31} = \lambda_{\sigma(3)\sigma(1)}$. \label{it:local11}
\item Corresponding points in these $T_3$s are points with the same barycentric coordinates: For $\mu_1, \mu_2, \mu_3 \in [0,1]$, $\sum_{i=1}^3 \mu_i = 1$, $\mu_1 e_1 + \mu_2 e_2 + \mu_3 e_3$ and $\mu_1 e'_{\sigma(1)} + \mu_2 e'_{\sigma(2)} + \mu_3 e'_{\sigma(3)}$ are corresponding points where $\sigma$ is the permutation in item~\eqref{it:local11} above.
\end{enumerate}
\end{definition}
 
We now prove \eqref{eq:T3-hull}:

\begin{proposition}
\label{prop:T3-hull}
Let $e_1, e_2, e_3 \in \Sym[c]{3}$ form a $T_3$. Then
\begin{equation*}
\RH(\{e_1, e_2, e_3\})
\supseteq [e_1,e_{1,2}] \cup [e_2,e_{2,3}] \cup [e_3,e_{3,1}] \cup \CH(\{e_{1,2}, e_{2,3}, e_{3,1}\}). \tag{\ref{eq:T3-hull}}
\end{equation*}
\end{proposition}

\begin{proof}
We use the same strategy as~\cite[Proposition 2.7]{Kreiner:2003}. Let $e_{1,2}$, $e_{2,3}$, $e_{3,1}$ be related to $e_1$, $e_2$, $e_3$ as in the left-hand side of Figures~\ref{fig:T3} and~\ref{fig:T3-hull}; the proof in the other case is similar.

Let $f$ be non-negative, symmetrised rank-one convex and vanish on $\{ e_1, e_2, e_3 \}$. Since $e_1 \compatible e_{1,2}$, it follows from the convexity of $f$ on $[e_1,e_{1,2}]$ that $f(e_{1,2}) \geqslant f(e_{3,1})$. Similarly, since $e_2 \compatible e_{2,3}$, it follows that $f(e_{2,3}) \geqslant f(e_{1,2})$. Finally from $e_3 \compatible e_{3,1}$ it follows that $f(e_{3,1}) \geqslant f(e_{2,3})$. In other words, $f(e_{2,3}) \geqslant f(e_{1,2}) \geqslant f(e_{3,1}) \geqslant f(e_{2,3})$. We conclude that $f(e_{2,3}) = f(e_{1,2}) = f(e_{3,1})= f(e_{2,3})$. But since $f(e_1)=0$, convexity of $f$ on $[e_1,e_{1,2}]$ shows that in fact $f(e_{1,2}) = f(e_{2,3}) = f(e_{3,1}) = 0$. From~\eqref{eq:sroch} we conclude that
\begin{equation*}
[e_1,e_{1,2}] \cup [e_2,e_{2,3}] \cup [e_3,e_{3,1}] \subset \RH(\{e_1,e_2,e_3\}).
\end{equation*}
The last step is to notice that since $e_{1,2}$, $e_{2,3}$, $e_{3,1}$ are pair-wise compatible, from Lemma~\ref{lem:pairwise-compatible}, 
\begin{equation*}
\CH(\{e_{1,2}, e_{2,3}, e_{3,1}\}) \subset \RH(\{e_1,e_2,e_3\}),
\end{equation*}
which completes the proof.
\end{proof}

We end this section by giving an example of the utility of two-dimensional results in higher dimensions.

\begin{lemma}[A three-dimensional continuum of $T_3$s]
\label{lem:T3-3D}
Let  $e_1, e_2, e_3 \in \Sym[c]{3}$ form a $T_3$. Let $e_0 \in \Sym[c]{3}$ such that $e_i \compatible e_0$ for $i=1,2,3$. For $\lambda \in \R$ and $i=1,2,3$ let
\begin{align*}
e_i^\lambda
 &:= \lambda e_0 + (1-\lambda)e_i.
\end{align*}
Then $e_1^\lambda, e_2^\lambda, e_3^\lambda$ also form a $T_3$ and 
\begin{equation} \label{eq:local12}
\bigcup_{\lambda \in [0,1]} \RH(\{e_1^\lambda, e_2^\lambda, e_3^\lambda\})
 \subseteq \RH(\{e_0,e_1,e_2,e_3\})
\end{equation}
which, when $e_0 \notin \aff \Span \{ e_1, e_2, e_3 \}$, is a three-dimensional continuum of $T_3$s.
\end{lemma}

\begin{figure}
\begin{center}
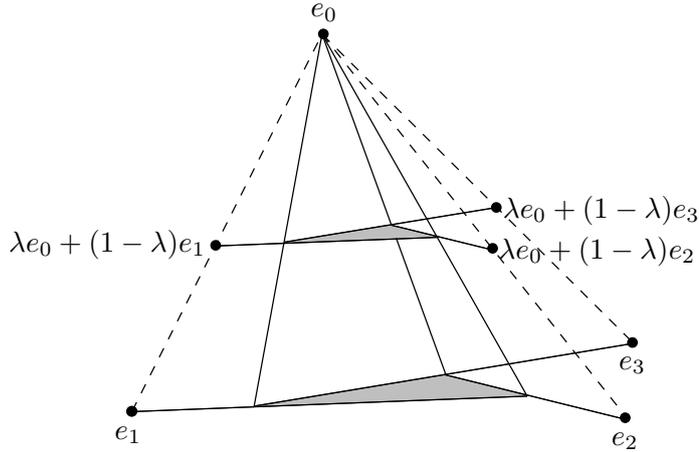
\caption{The three-dimensional continuum of $T_3$s described in Lemma~\ref{lem:T3-3D}.} 
\label{fig:3D-T3-continuum}
\end{center}
\end{figure}

\begin{proof}
Let $\lambda \in (0,1)$ and $i=1,2,3$.
Since $e_0 \compatible e_i$ it follows that $e_i^\lambda \in \LH(\{e_0,e_1,e_2,e_3\})$.
Since $\sign \det(e_i^\lambda - e_j^\lambda) = \sign \det (e_i - e_j)$, $i,j=1,2,3$, it follows that $e_1^\lambda, e_2^\lambda, e_3^\lambda$ form a $T_3$. The result follows.
\end{proof}

\section{The transformation strains of Monoclinic-I martensite} 
\label{sec:strains}
 
In this and the next section we prepare to apply the results of the preceding sections to monoclinic-I martensite by exploring first the symmetry and the geometry of this material.

We denote the transformation strains of the twelve variants of  cubic-to-monoclinic-I martensite, listed in Table~\ref{tab:strains}, by 
\begin{align*}
e^{(i)} \in \E
 &:= \{ e^{(i)} \ |\ i \in \I \}, \\
\I
 &:= \{1,2, \dots,12\} 
\end{align*} 
or simply by $i \in \I$ if the meaning is clear from the context. The transformation strains involve four lattice parameters $\alpha$, $\beta$, $\eps$ and $\delta$. These have been chosen such that $\eps > 0$ and $\delta > 0$. Typical lattice parameters are listed in Table~\ref{tab:parameters}. Note that $\Tr e =  2\alpha+\beta$ for $e \in \E$; thus $\E \subset \Sym[2\alpha+\beta]{3}$.

\begin{table}[t]
\begin{center}
\begin{tabular}{l|l||l|l||l|l||l|l}
 i & $e^{(i)}$ & i & $e^{(i)}$ & i & $e^{(i)}$ & i & $e^{(i)}$ \\ \hline \\[-2ex]
 1 & $\begin{pmatrix} \alpha & \delta & \eps \\
                     \delta & \alpha & \eps \\
                     \eps & \eps &\beta \end{pmatrix}$ 
 & 2 & $\begin{pmatrix} \alpha & \delta & -\eps \\
                     \delta & \alpha & -\eps \\
                     -\eps & -\eps &\beta \end{pmatrix}$
 & 3 & $\begin{pmatrix} \alpha & -\delta & -\eps \\
                     -\delta & \alpha & \eps \\
                     -\eps & \eps &\beta \end{pmatrix} $
 & 4 & $\begin{pmatrix} \alpha & -\delta & \eps \\
                     -\delta & \alpha & -\eps \\
                     \eps & -\eps &\beta \end{pmatrix}$ \\[4ex] \hline &&&&&& \\[-2ex]
 5 & $\begin{pmatrix} \alpha & \eps & \delta \\
                     \eps & \beta & \eps \\
                     \delta & \eps &\alpha \end{pmatrix}$ 
 & 6 & $\begin{pmatrix}\alpha & -\eps & \delta \\
                     -\eps & \beta & -\eps \\
                     \delta & -\eps &\alpha \end{pmatrix}$
 & 7 & $\begin{pmatrix} \alpha & -\eps & -\delta \\
                     -\eps & \beta & \eps \\
                     -\delta & \eps &\alpha \end{pmatrix}$ 
 & 8 & $\begin{pmatrix}\alpha & \eps & -\delta \\
                     \eps & \beta & -\eps \\
                     -\delta & -\eps &\alpha \end{pmatrix}$ \\[4ex] \hline &&&&&& \\[-2ex]
 9 & $\begin{pmatrix} \beta & \eps & \eps \\
                     \eps & \alpha & \delta \\
                     \eps & \delta &\alpha \end{pmatrix}$ 
 & 10 & $\begin{pmatrix} \beta & -\eps & -\eps \\
                     -\eps & \alpha & \delta \\
                     -\eps & \delta &\alpha \end{pmatrix}$
 & 11 & $\begin{pmatrix} \beta & -\eps & \eps \\
                     -\eps & \alpha & -\delta \\
                     \eps & -\delta &\alpha \end{pmatrix} $
 & 12 & $\begin{pmatrix} \beta & \eps & -\eps \\
                     \eps & \alpha & -\delta \\
                     -\eps & -\delta &\alpha \end{pmatrix}$ 
\end{tabular}
\caption{The transformation strains of the twelve variants of
  monoclinic-I martensite, cf.\ eg.\ \cite[Table 1, p.~119]{Bhattacharya:1997p99}.}
\label{tab:strains}
\end{center}
\end{table}

\begin{table}[t]
\begin{center}
\begin{tabular}{l|c|c|c|c|l}
 & $\alpha$ & $\beta$ & $\delta$ & $\eps$ & Reference\\\hline
 NiTi & 0.0243 & -0.0437 & 0.0580 & 0.0427 & \cite{Otsuka:1971,Knowles:1981,Hane:1999-07} \\ 
 CuZr & 0.0348 & 0.0229 & 0.1067 & 0.0929 & \cite{Seo:1998-02a,Seo:1998-02b} \\
 TiNiCu & 0.0232 & -0.0410 & 0.0532 & 0.0395 & \cite{Nam:1990}
\end{tabular}
\caption{Typical lattice parameters for monoclinic-I martensite, cf.\ eg.\ \cite[Table 2, p.~5459]{Shu:1998p5457} and \cite[p.~55 and 184]{Bhattacharya:2003}.}
\label{tab:parameters}
\end{center}
\end{table}

\paragraph{Compatibility.}
Compatibility and incompatibility between the transformation strains is of critical importance to us. A simple calculation~\cite[Compatibility.nb]{Chenchiah-Schloemerkemper1ESM} shows that for $e,f \in \E$, 
\begin{equation} \label{eq:determinant}
\det(e-f)
 \in \left\{ 0, \pm 4 \eps \left( (\alpha - \beta) \delta + \eps^2 - \delta^2 \right) \right\},
\end{equation}
which gives the compatibility/incompatibility of the strains by Lemma~\ref{lem:SymTr3-compatibility}. 

Note that if the material parameters happen to be such that $(\alpha-\beta) \delta + \eps^2 - \delta^2 = 0$ then all strains in $\E$ are pairwise compatible. Then by Lemma~\ref{lem:pairwise-compatible}, $\LH(\E) = \CH(\E)$. (In this case the material is able to form many more twins than usual~\cite{Pitteri:1998}.)

Pairs of compatible and incompatible transformation strains in $\E$ are listed in Table~\ref{tab:compatibility}. Here and henceforth (including in Mathematica calculations) we assume that $(\alpha-\beta) \delta + \eps^2 - \delta^2 \neq 0$ and, more generally, that the lattice parameters are generic. We also assume that $\alpha \neq \beta$, the (mathematical) reason for this will become clear in Section~\ref{sec:polytope}. There we will also see that the case $\eps = \delta$ is special so we shall specifically consider this possibility. 

\begin{table}[ht]
\begin{center}
\begin{tabular}{r|l|l|l}
 i & $\det(e^{(\cdot)} - e^{(i)}) = 0$ &  $\det(e^{(\cdot)} - e^{(i)}) \gtrless 0$ &  $\det(e^{(\cdot)} - e^{(i)}) \lessgtr 0$ \\ \hline
 1 & 2, 3, 4, 5, 7, 9, 11 & 8, 12 & 6, 10 \\
 2 & 1, 3, 4, 6, 8, 10, 12 & 5, 9 & 7, 11 \\
 3 & 1, 2, 4, 5, 7, 10, 12 & 6, 11 & 8, 9 \\
 4 & 1, 2, 3, 6, 8, 9, 11 & 7, 10 & 5, 12 \\
 5 & 1, 3, 6, 7, 8, 9, 12 & 4, 11 & 2, 10 \\
 6 & 2, 4, 5, 7, 8, 10, 11 & 1, 9 & 3, 12 \\
 7 & 1, 3, 5, 6, 8, 10, 11 & 2, 12 & 4, 9 \\
 8 & 2, 4, 5, 6, 7, 9, 12 & 3, 10 & 1, 11 \\
 9 & 1, 4, 5, 8, 10, 11, 12 & 3, 7 & 2, 6 \\
10 & 2, 3, 6, 7, 9, 11, 12 & 1, 5 & 4, 8 \\
11 & 1, 4, 6, 7, 9, 10, 12 & 2, 8 & 3, 5 \\
12 & 2, 3, 5, 8, 9, 10, 11 & 4, 6 & 1, 7
\end{tabular}
\caption{Compatible and incompatible transformation strains~\cite[Compatibility.nb]{Chenchiah-Schloemerkemper1ESM}. The signs in the second and third columns depend on the material parameters; the sign in the third column is opposite to the one in the second.}
\label{tab:compatibility}
\end{center}
\end{table}

\paragraph{Distances.} 
Also of importance is the distance between the transformation strains~\cite[Distances.nb]{Chenchiah-Schloemerkemper1ESM}:

\begin{obs}
\label{obs:distance}
Every pair of incompatible transformation strains is equidistant: For $e,f \in \E$ with $e \incompatible f$,
\begin{subequations} \label{eq:distance}
\begin{equation} \label{eq:distance-incompatible}
\| e - f \|^2
 = 2 (\alpha - \beta)^2 + 4 \delta^2 + 12 \epsilon^2.
\end{equation}
Remark~\ref{rem:incompatible-distances} sheds more light on this.

However for pairs of compatible transformation strains the situation is more complex: For $e,f \in \E$ with $e \compatible f$ and $e \neq f$,
\begin{equation} \label{eq:distance-compatible-1}
\| e - f \|^2
 \in  \left\{ 16 \eps^2, 8 (\delta^2 + \eps^2), 2 (\alpha - \beta)^2 + 4 (\delta - \eps)^2, 2 (\alpha - \beta)^2 + 4 (\delta + \eps)^2 \right\}.
\end{equation}
\end{subequations}
Table~\ref{tab:distances} presents the full picture. (See also Remark~\ref{rem:compatible-distances}.)
\end{obs}

We exclude (until Section~\ref{sec:conclusions}) the special case $\eps = \delta$ for which, for $e,f \in \E$ with $e \neq f$,
\begin{equation*}
\| e - f \|^2
 \in  \left\{ 16 \eps^2, 2 (\alpha - \beta)^2, 2 (\alpha - \beta)^2 + 16 \eps^2 \right\}.
\end{equation*}
Note that in this case, we cannot anymore distinguish between compatible and incompatible strains on the basis of the distance between them.

\begin{table}[ht]
\begin{center}
\begin{tabular}{r|l|l|l|l}
 i & $\| e^{(\cdot)} - e^{(i)} \|^2$ & $\| e^{(\cdot)} - e^{(i)} \|^2$ & $\| e^{(\cdot)} - e^{(i)} \|^2$ & $\| e^{(\cdot)} - e^{(i)} \|^2$ \\
 & $= 16 \eps^2$ & $= 8 (\delta^2 + \eps^2)$ & $= 2 (\alpha - \beta)^2 + 4 (\delta - \eps)^2$ & $= 2 (\alpha - \beta)^2 + 4 (\delta + \eps)^2$ \\ \hline
 1 & 2 & 3, 4 & 5, 9 & 7, 11 \\
 2 & 1 & 3, 4 & 8, 12 & 6, 10 \\
 3 & 4 & 1, 2 & 7, 10 & 5, 12 \\
 4 & 3 & 1, 2 & 6, 11 & 8, 9 \\
 5 & 6 & 7, 8 & 1, 9 & 3, 12 \\
 6 & 5 & 7, 8 & 4, 11 & 2, 10 \\
 7 & 8 & 5, 6 & 3, 10 & 1, 11 \\
 8 & 7 & 5, 6 & 2, 12 & 4, 9 \\
 9 & 10 & 11, 12 & 1, 5 & 4, 8 \\
10 & 9 & 11, 12 & 3, 7 & 2, 6 \\
11 & 12 & 9, 10 & 4, 6 & 1, 7 \\
12 & 11 & 9, 10 & 2, 8 & 3, 5 
\end{tabular}
\caption{Distances between compatible transformation strains~\cite[Distances.nb]{Chenchiah-Schloemerkemper1ESM}.}
\label{tab:distances}
\end{center}
\end{table}

\paragraph{Symmetry.} 
Now that we have knowledge of the compatibilities and the distances between the transformation strains we are ready to analyse the symmetry between them:

\begin{definition}[Symmetry and symmetry group] \label{def:symmetrygroups} \par \indent
\begin{enumerate}
\item A map  $\tau \colon \E \to \E$ is a symmetry of $\E$ if it preserves distance and compatibility in $\E$. That is, $\forall e,f \in \E$,
\begin{align*}
\| e - f \|
 &= \| \tau e - \tau f \|, \\
\det (e - f)
 &= \pm \det( \tau e - \tau f ).
\end{align*}
A symmetry group of $\E$ is a group of symmetries of $\E$.
\item Let $n \in \I \setminus \{1\}$ and ${\mathcal E}_n$ be a set of subsets of $\E$, all with cardinality $n$. That is,
\begin{equation*}
 {\mathcal E}_n \subset \left\{ S \subset \E\ |\ \# S = n \right\}.
\end{equation*}
A map $\tau \colon {\mathcal E}_n \to {\mathcal E}_n$ is a symmetry of ${\mathcal E}_n$ if it preserves distance and compatibility in ${\mathcal E}_n$. That is, $\forall S \in {\mathcal E}_n$, $\forall e,f \in S$,
\begin{align*}
\| e - f \|
 &= \| \tau e - \tau f \|, \\
\det (e - f)
 &= \pm \det( \tau e - \tau f ).
\end{align*}
A symmetry group of ${\mathcal E}_n$ is a group of symmetries of ${\mathcal E}_n$.
\end{enumerate}
\end{definition}

There are four sets that are of interest to us here. These are:
(i) $\E$ itself, 
(ii) $\E^2_{\compatible}$, the set of pairs of compatible transformation strains,
(iii) $\E^2_{\incompatible}$, the set of pairs of incompatible transformation strains, and
(iv) $\E^3_{\incompatible}$, the set of three-tuples of incompatible transformation strains:
\begin{align*}
\E^2_{\compatible}
 &:= \left\{ \{e,f\}\ |\ e,f \in \E,\ e \neq f,\ e \compatible f \right\}, \\
\E^2_{\incompatible}
 &:= \left\{ \{e,f\}\ |\ e,f \in \E,\ e \incompatible f \right\}, \\
\E^3_{\incompatible}
 &:= \left\{ \{e,f,g\}\ |\ e,f,g \in \E,\ e \incompatible f \incompatible g \incompatible e \right\}. 
\end{align*}
We characterise the symmetry of $\E$ and $\E^2_{\compatible}$ in Lemma~\ref{lem:symmetry1} and those of $\E^2_{\incompatible}$ and $\E^3_{\incompatible}$ in Lemma~\ref{lem:symmetry2} below. In order to do so we begin with some observations. For further investigation of these sets see~\cite{Chenchiah-Schloemerkemper-JMPS}. 

\begin{obs}[$S_4$ is a symmetry group of $\E$] \label{obs:oriented_cube}
Since the transformation strains in $\E$ are obtained through a phase transformation from a cubic crystal, $S_4$, the group of rotational symmetries of a cube, is a symmetry group of $\E$.
\end{obs}

This group and its action on $\E$ can be generated as follows: Let $R_1$, $R_2$ and $R_3$ be anticlockwise rotations of $\frac{\pi}{2}$ about the coordinate axes:
\begin{equation*}
R_1
 = \begin{pmatrix}
      1 & 0 & 0 \\
      0 & 0 & -1 \\
      0 & 1 & 0
     \end{pmatrix}, \qquad 
R_2
 = \begin{pmatrix}
      0 & 0 & 1 \\
      0 & 1 & 0 \\
      -1 & 0 & 0
     \end{pmatrix}, \qquad
R_3
 = \begin{pmatrix}
      0 & -1 & 0 \\
      1 & 0 & 0 \\
      0 & 0 & 1
     \end{pmatrix}.
\end{equation*}
For $i=1,2,3$, let $r_i$ be the map
\begin{equation} \label{eq:coordinate-rotations}
\Sym{3} \ni e \mapsto r_i e
 := R_i e R_i^T.
 \end{equation}
It is immediate that these are distance and determinant (and thus, compatibility) preserving: $\forall i=1,2,3$, $\forall e,f \in \E$,
\begin{subequations} \label{eq:rotations}
\begin{align}
\| e - f \|
 &= \| r_i e - r_i f \|, \\
\det (e - f)
 &= \det( r_i e - r_i f ).
\end{align}
\end{subequations}
Then $\{r_1,r_2,r_3\}$ generates $S_4$. (In fact any two of $r_1,r_2,r_3$ generate $S_4$ but it is convenient to retain all three.) The action of $S_4$ on $\E$ is listed, eg.\ in \cite[Table 1, p. 2607]{Hane:1999-07} but their numbering of the transformation strains is different from ours.

$S_4$ is isomorphic to a group of permutations on $\I$. We denote this group too by $S_4$; the images of $r_1$, $r_2$, $r_3$ under this isomorphism are also denoted by $r_1$, $r_2$, $r_3$. Table~\ref{tab:generators} lists the action of $r_1$, $r_2$, $r_3$ on $\I$~\cite[Symmetry.nb]{Chenchiah-Schloemerkemper1ESM}.

\begin{table}[t]
\begin{center}
\begin{tabular}{c|c|c|c|c|c|c|c|c|c|c|c|c}
 $i$ & 1 & 2 & 3 & 4 & 5 & 6 & 7 & 8 & 9 & 10 & 11 & 12 \\ \hline 
 $r_1 i$ & 6 & 5 & 8 & 7 & 4 & 3 & 2 & 1 & 11 & 12 & 10 & 9 \\ \hline
 $r_2 i$ & 12 & 11 & 9 & 10 & 8 & 7 & 5 & 6 & 2 & 1 & 3 & 4 \\ \hline
 $r_3 i$ & 3 & 4 & 2 & 1 & 10 & 9 & 12 & 11 & 7 & 8 & 5 & 6
\end{tabular}
\caption{The action of $r_1$, $r_2$, $r_3$ on $\I$.}
\label{tab:generators}
\end{center}
\end{table}

\begin{remark} \label{rem:cube}
Since $S_4$ is the group of rotational symmetries of a cube, it is natural and convenient to identify $\E$ (and $\I$) with the edges of a cube as shown in Figure~\ref{fig:oriented-cube}. (We could also have identified them with the diagonals of the faces.) Then $r_1$, $r_2$, $r_3$ are anticlockwise rotations of $\frac{\pi}{2}$ along an axis perpendicular to the face of the cube and passing through its centre.
 
From Table~\ref{tab:compatibility} we note that the four edges with which an edge is incompatible are precisely the four edges with which it shares a vertex. Thus,
\begin{enumerate}
\item \label{it:corners-of-faces} The 24 elements of $\E^2_{\incompatible}$ can be identified with the 24 corners of the faces of a cube, and
\item \label{it:corners}The 8 elements of $\E^3_{\incompatible}$ can be identified with the 8 corners of a cube.
\end{enumerate}
\end{remark}

\begin{figure}[ht]
\begin{center}
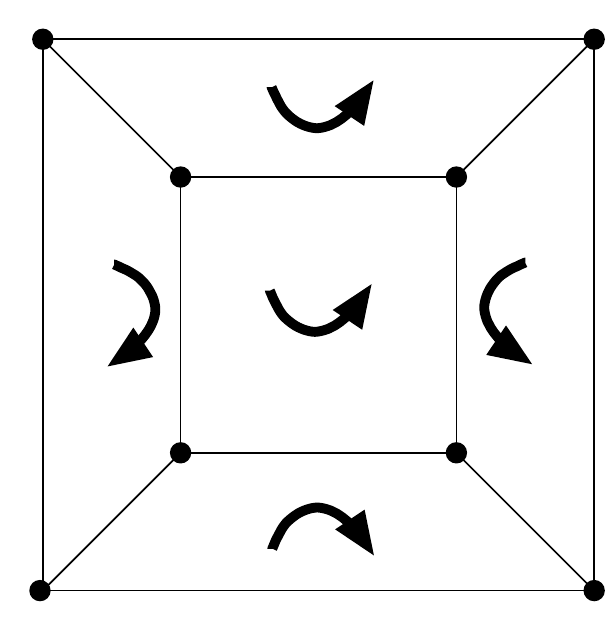 \quad \qquad
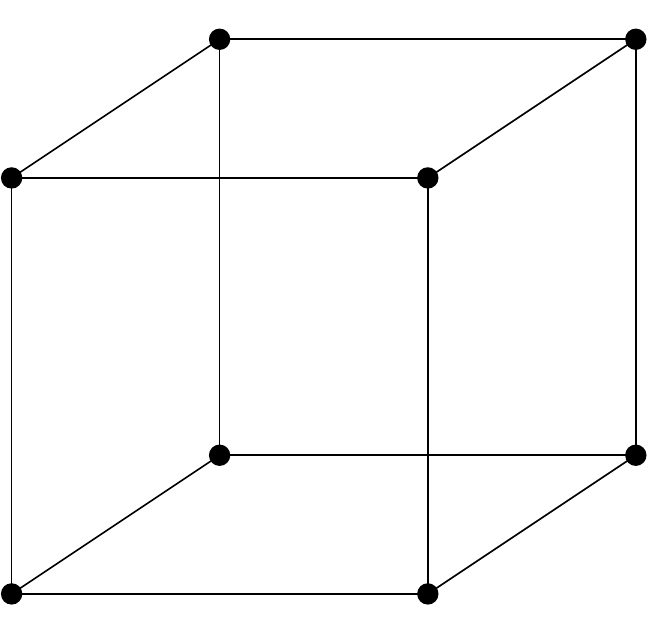
\caption{The oriented cube representing the symmetries of $\E$~\cite[Symmetry.nb]{Chenchiah-Schloemerkemper1ESM}.}
\label{fig:oriented-cube}
\end{center}
\end{figure}

Our next task is to show that $\E$ does not have the reflection symmetries of a cube. We will accomplish this in Observation~\ref{obs:inversion-distance} below.

\begin{definition}[Inversion] \label{def:inversion}
Let $r_0$ be the permutation on $\I$ which interchanges $1$ and $2$, $3$ and $4$, $5$ and $6$, $7$ and $8$, $9$ and $10$, and $11$ and $12$. 
\end{definition}

As can be seen from Figure~\ref{fig:oriented-cube}, $r_0$ is an inversion (reflection through the centre) of the cube. This immediately shows both that $r_0 \notin S_4$ and that $\{ r_0, r_1, r_2, r_3 \}$ generates $S_4 \times C_2$, the group of symmetries of a cube (including reflections).

\begin{remark} \label{rem:inversion}
The permutation $r_0$ corresponds to replacing $\eps$ with $-\eps$ in the transformation strains (cf.\ Table~\ref{tab:strains}).
\end{remark}

\begin{obs}
Unlike $r_1$, $r_2$ and $r_3$, $r_0$ cannot be identified with a linear operator on $\Sym[2\alpha+\beta]{3}$.
\end{obs}

\begin{proof}
Assume, on the contrary, that there exists $\L \colon \Sym[2\alpha+\beta]{3} \to \Sym[2\alpha+\beta]{3}$, linear, such that $\forall i \in \I$, $\mathcal{L} e^{(i)} = e^{(r_0 i)}$. Then,
\begin{alignat*}{2}
\L \left( e^{(1)} + e^{(2)} + e^{(3)} + e^{(4)} \right)
 &= e^{(1)} + e^{(2)} + e^{(3)} + e^{(4)} \\
  &\implies \L \begin{pmatrix} \alpha & 0 & 0 \\ 0 & \alpha & 0 \\ 0 & 0 & \beta \end{pmatrix}
   =  \begin{pmatrix} \alpha & 0 & 0 \\ 0 & \alpha & 0 \\ 0 & 0 & \beta \end{pmatrix}, \\
\L \left( e^{(5)} + e^{(6)} + e^{(7)} + e^{(8)} \right)
 &= e^{(5)} + e^{(6)} + e^{(7)} + e^{(8)} \\
  &\implies \L \begin{pmatrix} \alpha & 0 & 0 \\ 0 & \beta & 0 \\ 0 & 0 & \alpha \end{pmatrix}
   =  \begin{pmatrix} \alpha & 0 & 0 \\ 0 & \beta & 0 \\ 0 & 0 & \alpha \end{pmatrix}, \\
\L \left( e^{(9)} + e^{(10)} + e^{(11)} + e^{(12)} \right)
 &= e^{(9)} + e^{(10)} + e^{(11)} + e^{(12)} \\
  &\implies \L \begin{pmatrix} \beta & 0 & 0 \\ 0 & \alpha & 0 \\ 0 & 0 & \alpha \end{pmatrix}
   =  \begin{pmatrix} \beta & 0 & 0 \\ 0 & \alpha & 0 \\ 0 & 0 & \alpha \end{pmatrix}.
\end{alignat*}
Thus (i) $\L$ is an identity on the diagonal components of $\Sym[2\alpha+\beta]{3}$. Moreover
\begin{alignat*}{4}
\L \left( e^{(1)} + e^{(3)} \right)
 &= e^{(2)} + e^{(4)}
 &\implies \L \begin{pmatrix} \alpha & 0 & 0 \\ 0 & \alpha & \eps \\ 0 & \eps & \beta \end{pmatrix}
 &=  \begin{pmatrix} \alpha & 0 & 0 \\ 0 & \alpha & -\eps \\ 0 & -\eps & \beta \end{pmatrix}, \\
\L \left( e^{(5)} + e^{(8)} \right)
 &= e^{(6)} + e^{(7)}
 &\implies \L \begin{pmatrix} \alpha & \eps & 0 \\ \eps & \beta & 0 \\ 0 & 0 & \alpha \end{pmatrix}
 &=  \begin{pmatrix} \alpha & -\eps & 0 \\ -\eps & \beta & 0 \\ 0 & 0 & \alpha \end{pmatrix}, \\
\L \left( e^{(9)} + e^{(11)} \right)
 &= e^{(10)} + e^{(12)}
 &\implies \L \begin{pmatrix} \beta & 0 & \eps \\ 0 & \alpha & 0 \\ \eps & 0 & \alpha \end{pmatrix}
 &=  \begin{pmatrix} \beta & 0 & -\eps \\ 0 & \alpha & 0 \\ -\eps & 0 & \alpha \end{pmatrix}.
\end{alignat*}
Thus (ii) $\L$ is a negative of the identity on the off-diagonal components of $\Sym[2\alpha+\beta]{3}$.

It is easy to check that (i) and (ii) are contradictory, eg. $\L e^{(1)} = e^{(2)}$. 
\end{proof}

Now we address the question of whether $r_0$ is a symmetry of $\E$:

\begin{obs}[Inversions preserve compatibility] \label{obs:inversion-compatibility}
As can be verified~\cite[Symmetry.nb]{Chenchiah-Schloemerkemper1ESM} from Table~\ref{tab:compatibility} and~\eqref{eq:determinant}, $r_0$ is compatibility preserving: $\forall e,f \in \E$,
\begin{equation*} 
\det (e - f)
 = - \det( r_0 e - r_0 f ).
\end{equation*}
\end{obs}

\begin{obs}[Effect of inversions on distance] \label{obs:inversion-distance}
From Observation~\ref{obs:distance} and Remark~\ref{rem:inversion} it follows that $r_0$ is distance-preserving on pairs of incompatible transformation strains. However it is distance-preserving only on some pairs of compatible transformation strains: For $e,f \in \E$,
\begin{equation} \label{eq:inversion-distance}
\| e - f \|^2 = \| r_0 e - r_0 f \|^2
 \qquad \iff \qquad
  \| e - f \|^2 \neq 2 (\alpha - \beta)^2 + 4 (\delta \pm \eps)^2.
\end{equation}
This is easy to verify in view of the following: Let $e,f \in \E$ with $e \compatible f$. Then,
\begin{equation} \label{eq:distance-compatible-2}
\| e - f \|^2 =
 \begin{cases}
  16 \eps^2 &\text{ if } e = r_0 f, \\ 
  8 \delta^2 + 8 \eps^2 &\text{ if } \exists i \in \{1,2,3\},\ e = r_i f, \\
  2 (\alpha - \beta)^2 + 4 (\delta \pm \eps)^2 &\text{ else}.
 \end{cases}
\end{equation}
(The middle case corresponds to parallel edges of the cube that share a face.)
\end{obs}

Thus $r_0$ is not a symmetry of $\E$ nor is it a symmetry of $\E^2_{\compatible}$. In the light of Observation~\ref{obs:oriented_cube} we conclude that $\E$ and $\E^2_{\compatible}$ have the rotation symmetries of a cube but not its reflection symmetries:

\begin{lemma}
\label{lem:symmetry1}
$S_4$ is a symmetry group of $\E$ and $\E^2_{\compatible}$, whereas $S_4 \times C_2$ is not a symmetry group of  $\E$ or $\E^2_{\compatible}$.
\end{lemma}

\begin{remark} \label{rem:compatible-distances}
From~\eqref{eq:distance-compatible-1} we deduce that the orbits of $S_4$ partition $\E^2_{\compatible}$ into four equivalence classes. (Equation~\eqref{eq:distance-compatible-2} relates this to the edges of a cube.)
\end{remark}

Observations~\ref{obs:distance}, \ref{obs:oriented_cube} and~\ref{obs:inversion-compatibility} show that $\E^2_{\incompatible}$ and $\E^3_{\incompatible}$ have the symmetries of a cube:

\begin{lemma}
\label{lem:symmetry2}
$S_4 \times C_2$ is a symmetry group of $\E^2_{\incompatible}$ and of $\E^3_{\incompatible}$.
\end{lemma}

\begin{remark} \label{rem:incompatible-distances}
From Remark~\ref{rem:cube}\eqref{it:corners-of-faces} it is clear that $S_4 \times C_2$ (or indeed $S_4$) acts transitively on $\E^2_{\incompatible}$, i.e., the orbit of any element of   $\E^2_{\incompatible}$ under $S_4 \times C_2$ equals $\E^2_{\incompatible}$. This explains Observation~\ref{obs:distance}.
Likewise, from Remark~\ref{rem:cube}\eqref{it:corners} it is clear that $S_4 \times C_2$ (or indeed $S_4$) acts transitively on $\E^3_{\incompatible}$.
\end{remark}

\section{The convex polytope formed by the transformation strains}
\label{sec:polytope}

In this section we study $\CH(\E)$, the convex hull of $\E$, which is a five-dimensional polytope~\cite[Dimension.nb]{Chenchiah-Schloemerkemper1ESM}, when, as we have assumed, $\alpha \neq \beta$ (otherwise it is a three-dimensional polytope).  We are interested in the \emph{facets} of $\CH(\E)$. However the \emph{vertices} and \emph{edges} of $\CH(\E)$ are also of interest and we begin with them.

For convenience we set $\Lambda = \{ \lambda \in [0,1]^{12},\ \sum_{i=1}^{12} \lambda_i = 1 \}$. The following linear functionals will  be helpful in studying $\CH(\E)$:

\begin{definition}
The linear functionals $H_i \colon \Sym{3} \to \R$, $i=0,1,2,3$, are defined by 
\begin{alignat*}{2}
 &\phantom{==}H_0\, e 
 &&= -e_{12} - e_{23} - e_{31}, \\
H_1\, e 
 &= H_0\, r_1 e
  &&= -e_{12} + e_{23} + e_{31}, \\
H_2\, e 
 &= H_0\, r_2 e
 &&= e_{12} - e_{23} + e_{31}, \\
H_3\, e 
 &= H_0\, r_3 e
 &&= e_{12} + e_{23} - e_{31}; 
\end{alignat*}
and $H_{ij} \colon \Sym{3} \to \R$, $i,j=1,2,3$, by
\begin{equation*}
H_{ij}\, e
 = e_{ij},
\end{equation*}
where $e_{ij}$ denotes the $(i,j)$-component of the matrix $e$.
\end{definition}
For the convenience of the reader we summarise the images of the transformation strains under the functionals $H_i$, $i=0,1,2,3$, in Table~\ref{tab:H}. 
We also list the extremisers of these functionals; note that the extremisers when $\eps = \delta$ are the union of the extremisers when $\eps < \delta$ and when $\eps > \delta$. 

\begin{table}[t]
\begin{center}
\begin{tabular}{c|c|c|c|c}
i & $H_0\, e^{(i)}$ & $H_1\, e^{(i)}$ & $H_2\, e^{(i)}$ & $H_3\, e^{(i)}$  \\ \hline \hline
1 & $-\delta - 2 \eps$ & $-\delta + 2 \eps$ & $\delta$ & $\delta$ \\ \hline
2 & $-\delta + 2 \eps$ & $-\delta - 2 \eps$ & $\delta$ & $\delta$ \\ \hline
3 & $\delta$ & $\delta$ & $-\delta - 2 \eps$ & $-\delta + 2 \eps$ \\ \hline
4 & $\delta$ & $\delta$ & $-\delta + 2 \eps$ & $-\delta - 2 \eps$ \\ \hline
5 & $-\delta - 2 \eps$ & $\delta$ & $\delta$ & $-\delta + 2 \eps$ \\ \hline
6 & $-\delta + 2 \eps$ & $\delta$ & $\delta$ & $-\delta - 2 \eps$  \\ \hline
7 & $\delta$ & $-\delta + 2 \eps$ & $-\delta - 2 \eps$ & $\delta$ \\ \hline
8 & $\delta$ & $-\delta - 2 \eps$ & $-\delta + 2 \eps$ & $\delta$ \\ \hline
9 & $-\delta - 2 \eps$ & $\delta$ & $-\delta + 2 \eps$ & $\delta$ \\ \hline
10 & $-\delta + 2 \eps$ & $\delta$ & $-\delta - 2 \eps$ & $\delta$ \\ \hline
11 & $\delta$ & $-\delta + 2 \eps$ & $\delta$ & $-\delta - 2 \eps$  \\ \hline
12 & $\delta$ & $-\delta - 2 \eps$ & $\delta$ & $-\delta + 2 \eps$
\end{tabular}

\bigskip

\begin{tabular}{c|c|c|c|c}
 & $i \mapsto H_0\, e^{(i)}$ & $i \mapsto H_1\, e^{(i)}$ & $i \mapsto H_2\, e^{(i)}$ & $i \mapsto H_3\, e^{(i)}$ \\ \hline \hline
minimisers & 1,5,9 & 2,8,12 & 3,7,10 & 4,6,11 \\ \hline
\end{tabular}

\bigskip

\begin{tabular}{c|c|c|c|c}
Maximisers & $i \mapsto H_0\, e^{(i)}$ & $i \mapsto H_1\, e^{(i)}$ & $i \mapsto H_2\, e^{(i)}$ & $i \mapsto H_3\, e^{(i)}$ \\ \hline \hline
When $\eps < \delta$ & 3,4,7,8,11,12 & 3,4,5,6,9,10 & 1,2,5,6,11,12 & 1,2,7,8,9,10 \\ \hline
When $\eps = \delta$ & $\I \setminus \{1,5,9\}$ & $\I \setminus \{2,8,12\}$ & $\I \setminus \{3,7,10\}$ & $\I \setminus \{4,6,11\}$ \\ \hline
When $\eps > \delta$ & 2,6,10 & 1,7,11 & 4,8,9 & 3,5,12
\end{tabular}
\caption{The images of the transformation strains under the linear functionals $H_i$, $i=0,1,2,3$, along with their extremisers~\cite[Linear\_Functionals.nb]{Chenchiah-Schloemerkemper1ESM}.}
\label{tab:H}
\end{center}
\end{table}

\paragraph{Vertices.}
The \emph{vertices} of $\CH(\E)$ are the transformation strains:

\begin{lemma} \label{lem:verticesofC}
The set of vertices of $\CH(\E)$ is $\E$.
\end{lemma}

\begin{proof}
We show this explicitly for $e^{(1)}$, the proof for the other vertices follows from symmetry.

Let $\lambda \in \Lambda$ such that $e^{(1)} = \sum_{i=1}^{12} \lambda_i e^{(i)}$. Then,
\begin{equation*}
 \beta = H_{33} e^{(1)}  = \sum_{i=1}^{12} \lambda_i H_{33}  e^{(i)} = \sum_{i=1}^{4} \lambda_i \beta + \sum_{i=5}^{12} \lambda_i \alpha,
\end{equation*}
from which we conclude that $\lambda_i = 0$ for $i=5,\dots,12$ and thus $e^{(1)} = \sum_{i=1}^{4} \lambda_i e^{(i)}$.
Now, using $H_0$ we obtain
\begin{equation*}
 -\delta - 2\epsilon = H_0 e^{(1)} = \sum_{i=1}^{4} \lambda_i H_0 e^{(i)}.
\end{equation*}
Since $\eps, \delta > 0$ it is easy to see (cf.\ Table~\ref{tab:H}) that $\lambda_1 = 1$ and $\lambda_i = 0$, $i=2,3,4$. We conclude that $e^{(1)}$ is a vertex.
\end{proof}

In the interest of brevity in future proofs of extremality of subsets we will only name the relevant family of four-dimensional hyperplanes, eg.\ for the above lemma we would say that this follows from $H_{33}$ and $H_0$.

\paragraph{Edges.}
Contrary to what we are used to in two and three dimensions (see Remark~\ref{rem:local1} below), the convex hull of \emph{every} pair of vertices is an edge of $\CH(\E)$:

\begin{lemma} \label{lem:edges}
The set of edges of $\CH(\E)$ is $\{ [e,f]\ |\ e,f \in \E \}$.
\end{lemma}

Henceforth, a \emph{compatible edge} is the convex hull of a pair of compatible vertices, and an \emph{incompatible edge} is the convex hull of a pair of incompatible vertices.

We prove Lemma~\ref{lem:edges} here except for incompatible edges when $\eps > \delta$. In Remark~\ref{rem:edges} we present a Mathematica-aided proof which is valid for $\eps \neq \delta$.

\begin{proof}[for compatible edges and, when $\eps \leqslant \delta$, for incompatible edges]
By symmetry it suffices to prove that the eleven edges $[e^{(1)},e^{(i)}]$, $i \in \I \setminus \{1\}$, are extremal.

This is easy to verify for the compatible edges: For example,  $H_{33}$ and $H_{12}$ show that $[e^{(1)},e^{(2)}]$ is extremal. The proof for the other compatible edges (with $e^{(1)}$ as a vertex) is similar. 

We now turn to the incompatible edges, for example $[e^{(1)},e^{(6)}]$. Let $\mu \in [0,1]$ and $\lambda \in \Lambda$ such that
\begin{equation*}
\mu e^{(1)} + (1-\mu) e^{(6)}
 = \sum_{i=1}^{12} \lambda_i e^{(i)}.
\end{equation*}

Consider first the case $\eps < \delta$. Then $H_{11}$ and $H_2$ show that in fact
\begin{equation*}
\mu e^{(1)} + (1-\mu) e^{(6)}
 = \sum_{i \in \{1,2,5,6\}} \lambda_i e^{(i)}.
\end{equation*}
However $\dim \aff \Span\{ e^{(1)},e^{(2)},e^{(5)},e^{(6)} \} = 3$~\cite[Dimension\_Calculations.nb]{Chenchiah-Schloemerkemper1ESM} and thus \\ $\CH(\{e^{(1)},e^{(2)},e^{(5)},e^{(6)}\})$ is a three-dimensional tetrahedron. It follows that $\lambda_5 = \lambda_2 = 0$ and thus $[e^{(1)},e^{(6)}]$ is extremal. 

Next consider the case $\eps = \delta$. Then $H_{11}$, $H_{13}$, $H_1$ and $H_2$ show that in fact
\begin{equation*}
\mu e^{(1)} + (1-\mu) e^{(6)}
 = \sum_{i \in \{1,5,6\}} \lambda_i e^{(i)}.
\end{equation*}
However $\dim \aff \Span\{ e^{(1)},e^{(5)},e^{(6)} \} = 2$~\cite[Dimension\_Calculations.nb]{Chenchiah-Schloemerkemper1ESM} and thus \\ $\CH(\{e^{(1)},e^{(5)},e^{(6)}\})$ is a triangle. It follows that $\lambda_5 = 0$ and thus $[e^{(1)},e^{(6)}]$ is extremal. 

The extremality of the other incompatible edges (with $e^{(1)}$ as a vertex) follows from symmetry. 
\end{proof}

\begin{remark}
\label{rem:local1}
In dimensions less than four the only polytopes for which the convex hull of every pair of vertices is an edge are the $n$-tetrahedra (line segments, triangles and tetrahedra in dimensions one, two and three, respectively). However for every $n>3$ and every $d>n$ there exists an n-dimensional convex polytope with d vertices for which the convex hull of every pair of vertices is extremal. See, eg.\ \cite[Section 13]{Brondsted:1982}, \cite[Section 3]{Ewald:1996} or \cite[Corollary 0.8]{Ziegler:1994}.
\end{remark}

\paragraph{The facets of $\CH(\E)$.}

The algorithm we use to determine the facets of $\CH(\E)$ is as follows~\cite[Faceting.nb]{Chenchiah-Schloemerkemper1ESM}. It assumes that the affine span of the set is five-dimensional, that the cardinality of the set is small and thus that computational efficiency is not a consideration.

\begin{algorithm} \label{alg:facets}
\par \indent
\begin{enumerate}
\item First we form a set of all four-dimensional tetrahedra with vertices in $\E$ as follows:
	\begin{enumerate}
	\item Pick all five-tuples from $\E$:
	\begin{equation*}
	\{ S \subset \E\ |\ \#S = 5 \}.
	\end{equation*}
	\item Discard those five-tuples whose affine span is not four dimensional:
	\begin{equation*}
	\G_1 := \{ S \subset \E\ |\ \#S = 5, \dim \aff \Span(S) = 4 \}.
	\end{equation*}
	\end{enumerate}
\item Of these tetrahedra we discard those whose convex hull is not contained in $\partial \CH(\E)$. We do this as follows:
	\begin{enumerate}
	\item Let $\G_2 = \G_1$.
	\item \label{it:repeatS1} Pick $S \in \G_2$.
	\item Translate the origin to some $e \in S$. (The next two steps are carried out in this co-ordinate system.)
	\item Compute a normal $n \in \Sym[2\alpha+\beta]{3}$ to $\aff \Span(S)$.
	\item \label{it:repeatE1} If $\langle n, e \rangle$ has the same sign for all $e \in \E \setminus S$ then $\CH(S) \subset \partial \CH(\E)$. Otherwise remove $S$ from $G_2$
	\item Repeat steps~\eqref{it:repeatS1} to~\eqref{it:repeatE1} till all tetrahedra in $\G_2$ have been tested. 
	\end{enumerate}
	We now obtain 
	\begin{equation*}
	\G_2 = \{ S \subset \E\ |\ \#S = 5, \dim \aff \Span(S) = 4, \CH(S) \subset \partial \CH(\E) \}.
	\end{equation*}
	This is the set of all four-dimensional tetrahedra (with vertices in $\E$) whose union is $\partial \CH(\E)$.
\item The final step is to form the facets of $\CH(\E)$ by judiciously taking unions of sets in $\G_2$ as follows:
	\begin{enumerate}
	\item Let $\G_3 = \G_2$.
	\item \label{it:repeatS2} Pick $S_1,S_2 \in \G_3$.
	\item \label{it:repeatE2} If $\dim \aff \Span (S_1 \cup S_2) = 4$ then $S_1$ and $S_2$ are parts of the same facet. In $\G_3$, replace $S_1$ and $S_2$ by $S_1 \cup S_2$.
	\item Repeat steps~\eqref{it:repeatS2} and~\eqref{it:repeatE2} until every union of sets in $G_3$ increases the dimension, i.e., until it is true that
	\begin{equation} \label{eq:local1}
	\forall S_1, S_2 \in \G_3, \qquad \dim \aff \Span(S_1 \cup S_2) =4 \implies S_1 = S_2.
	\end{equation} 
	\end{enumerate}
	This is the set of all the facets of $\CH(\E)$. Note that it is a set of $n$-tuples where $n \geqslant 5$.
\end{enumerate}
\end{algorithm}

The results of a Mathematica implementation of Algorithm~\ref{alg:facets} are summarised in Observations~\ref{obs:e<d}, \ref{obs:e=d} and~\ref{obs:e>d} below. These reveal that the facet structure depends on whether $\eps < \delta$, $\eps = \delta$ or $\eps > \delta$. All three possibilities are realisable in that there exist cubic-to-monoclinic-I phase transformations corresponding to each. (See, e.g., \cite[Fig.~4.3 and~(4.11) on p.~52--53]{Bhattacharya:2003} for the relationship between $\eps$, $\delta$ and the unit cells of the cubic and monoclinic lattices.) However, curiously $\eps < \delta$ for all the monoclinic-I materials of which we are aware, cf.\ Table~\ref{tab:parameters}; we return to this point in Section~\ref{sec:conclusions}.

In the observations below, each group of facets is the orbit under $S_4$ of any facet in it~\cite[Facet\_Symmetry.nb]{Chenchiah-Schloemerkemper1ESM}. Within each group the facets are listed in lexical order. Facets that occur for both $\eps < \delta$ and $\eps > \delta$ are shown in bold face.

\begin{obs}[Monoclinic-Ia martensite, $\eps < \delta$]
\label{obs:e<d}
When $\eps < \delta$, the 25 four-dimensional facets of $\CH(\E)$ consist of the convex hulls of

\begin{enumerate}
\item 12 facets with 5 vertices each: \label{it:local12}
 \begin{alignat*}{4}
 & \{1, 2, 3, 7, 10\}, \
 && \{1, 2, 4, 6, 11\}, \
 && \{1, 3, 4, 5, 9\}, \
 && \{1, 5, 7, 8, 9\}, \\
 & \{1, 5, 9, 11, 12\}, \
 && \{2, 3, 4, 8, 12\}, \
 && \{2, 5, 6, 8, 12\}, \
 && \{2, 8, 9, 10, 12\},\\
 & \{3, 5, 6, 7, 10\}, \
 && \{3, 7, 10, 11, 12\}, \
 && \{4, 6, 7, 8, 11\}, \
 && \{4, 6, 9, 10, 11\};
 \end{alignat*}
\item 4 pairs of $T_3$s (see Section~\ref{sec:nonlaminates}), each facet is invariant under $r_0$:
  \begin{equation} \label{eq:T3-pair}
  \{1, 2, 5, 6, 11, 12\}, \ \{1, 2, 7, 8, 9, 10\}, \ \{3, 4, 5, 6, 9, 10\}, \ \{3, 4, 7, 8, 11, 12\}
  \end{equation}
\item 6 pairs of pairwise compatible three-tuples: \label{it:local13}
 \begin{alignat*}{3}
 & \mathbf{\{1, 2, 5, 8, 9, 12\}}, \
 && \mathbf{\{1, 3, 5, 7, 9, 10\}}, \
 && \mathbf{\{1, 4, 5, 6, 9, 11\}}, \\
 & \mathbf{\{2, 3, 7, 8, 10, 12\}},\
 && \mathbf{\{2, 4, 6, 8, 11, 12\}}, \
 && \mathbf{\{3, 4, 6, 7, 10, 11\}};
 \end{alignat*}
\item 3 facets with 8 vertices each; each facet is invariant under $r_0$:
 \begin{equation} \label{eq:8vertex-facets}
 \mathbf{\{1, 2, 3, 4, 5, 6, 7, 8\}}, \ \mathbf{\{1, 2, 3, 4, 9, 10, 11, 12\}}, \ \mathbf{\{5, 6, 7, 8, 9, 10, 11, 12\}}.
 \end{equation}
\end{enumerate}
\end{obs}

\begin{obs}[$\eps = \delta$]
\label{obs:e=d}
When $\eps = \delta$, the 7 four-dimensional facets of $\CH(\E)$ consist of the convex hulls of
\begin{enumerate}
\item 4 facets with 9 vertices each:
 \begin{alignat}{2}
 & \{1, 2, 3, 5, 7, 8, 9, 10, 12\}, \
 && \{1, 2, 4, 5, 6, 8, 9, 11, 12\}, \notag \\
 & \{1, 3, 4, 5, 6, 7, 9, 10, 11\}, \
 && \{2, 3, 4, 6, 7, 8, 10, 11, 12\}; \label{eq:9vertex-facets}
 \end{alignat}
\item 3 facets with 8 vertices each; each facet is invariant under $r_0$:
 \begin{equation}
 \mathbf{\{1, 2, 3, 4, 5, 6, 7, 8\}}, \ \mathbf{\{1, 2, 3, 4, 9, 10, 11, 12\}}, \ \mathbf{\{5, 6, 7, 8, 9, 10, 11, 12\}}.\tag{\ref{eq:8vertex-facets}}
 \end{equation}
\end{enumerate}
\end{obs}

\begin{obs}[Monoclinic-Ib martensite, $\eps > \delta$]
\label{obs:e>d}
When $\eps > \delta$, the 33 four-dimensional facets of $\CH(\E)$ consist of the convex hulls of 
\begin{enumerate}
\item 12 facets which are the images under $r_0$ of the five-vertex facets that occur when $\eps < \delta$:
  \begin{alignat*}{4}
  & \{1, 2, 3, 5, 12\}, \
  && \{1, 2, 4, 8, 9\}, \
  && \{1, 3, 4, 7, 11\}, \
  && \{1, 5, 6, 7, 11\}, \\
  & \{1, 7, 9, 10, 11\}, \
  && \{2, 3, 4, 6, 10\}, \
  && \{2, 6, 7, 8, 10\}, \
  && \{2, 6, 10, 11, 12\}, \\
  & \{3, 5, 7, 8, 12\}, \
  &&\{3, 5, 9, 10, 12\}, \
  && \{4, 5, 6, 8, 9\}, \
  && \{4, 8, 9, 11, 12\};
  \end{alignat*}
\item 12 other five vertex facets (together these are invariant under $r_0$):
 \begin{alignat*}{4} 
  & \{1, 3, 5, 9, 12\}, \
 && \{1, 3, 7, 10, 11\}, \
 && \{1, 4, 5, 8, 9\}, \
 && \{1, 4, 6, 7, 11\}, \\
 & \{1, 5, 7, 9, 11\}, \
 && \{2, 3, 5, 8, 12\}, \
 && \{2, 3, 6, 7, 10\}, \
 && \{2, 4, 6, 10, 11\}, \\
 & \{2, 4, 8, 9, 12\}, \  
 && \{2, 6, 8, 10, 12\}, \
 && \{3, 5, 7, 10, 12\}, \
 && \{4, 6, 8, 9, 11\};
\end{alignat*}
\item 6 pairs of pairwise compatible three-tuples:
 \begin{alignat*}{3}
 & \mathbf{\{1, 2, 5, 8, 9, 12\}}, \
 && \mathbf{\{1, 3, 5, 7, 9, 10\}}, \
 && \mathbf{\{1, 4, 5, 6, 9, 11\}}, \\
 & \mathbf{\{2, 3, 7, 8, 10, 12\}},\
 && \mathbf{\{2, 4, 6, 8, 11, 12\}}, \
 && \mathbf{\{3, 4, 6, 7, 10, 11\}};
 \end{alignat*}
\item 3 facets with 8 vertices each; each facet is invariant under $r_0$:
 \begin{equation}
 \mathbf{\{1, 2, 3, 4, 5, 6, 7, 8\}}, \ \mathbf{\{1, 2, 3, 4, 9, 10, 11, 12\}}, \ \mathbf{\{5, 6, 7, 8, 9, 10, 11, 12\}}. \tag{\ref{eq:8vertex-facets}}
 \end{equation}
\end{enumerate}
\end{obs}

\begin{remark}
The extremality of facets that are invariant under $r_0$ can be verified as in the proof of Lemma~\ref{lem:verticesofC}: $H_{11}, H_{22}, H_{33}$ show that the facets in~\eqref{eq:8vertex-facets} are extremal. $H_i$, $i=0,1,2,3$, show that pairs of $T_3$s in~\eqref{eq:T3-pair} are extremal since $\eps < \delta$ implies $-\delta-2\eps < -\delta < \delta-2\eps < \delta < \delta+2\eps$.

In addition $H_i$, $i=0,1,2,3$, also show that the facets in~\eqref{eq:9vertex-facets} are extremal, cf.\ Table~\ref{tab:H}. For the remaining facets extremality can be verified through a computation of normals~\cite[Facet\_Normal.nb]{Chenchiah-Schloemerkemper1ESM}. \end{remark}

\begin{remark}[Proof of Lemma~\ref{lem:edges} when $\eps \neq \delta$] \label{rem:edges}
Observations~\ref{obs:e<d} and \ref{obs:e>d} lead to a proof of Lemma~\ref{lem:edges} when $\eps \neq \delta$: Observe that every edge is shared by at least four facets. (This is particularly easy to check for the incompatible edges: Each incompatible edge is contained in precisely one facet from each group.) It follows that every edge is extremal (cf.\ eg.\ \cite[Chapter 6]{Barvinok:2002} or the other references listed in Section~\ref{sec:convexity}). 
\end{remark}

\begin{remark}
We remark that this phenomenon of polytope facet structure depending on lattice parameters is not possible for the other martensites (i.e., cubic-to-tetragonal, cubic-to-trigonal and cubic-to-orthorhombic) because they form $n$-tetrahedra (for $n=2,3,5$, respectively) and thus their facet structure is fixed.
\end{remark}

\section{Non-laminate microstructures in $\RH(\E)$}
\label{sec:nonlaminates}

In this section we use the results of the preceding sections to derive our central results about non-laminate microstructures in monoclinic-I martensite. These include $T_3$ microstructures formed by the strains in $\E$ (Section~\ref{sec:T3s-level1}
) and $T_3$ microstructures formed by the nodes of these $T_3$s (Section~\ref{sec:T3s-level2}
).

\subsection{(Level-1) $T_3$s and related microstructures}
\label{sec:T3s-level1}

As mentioned earlier there are precisely eight 3-tuples of pairwise incompatible vertices (cf.\ Table~\ref{tab:compatibility}):
\begin{multline*}
\E^3_{\incompatible}
 = \left\{ \{ e^{(1)}, e^{(6)}, e^{(12)} \},\ \{ e^{(1)}, e^{(8)}, e^{(10)} \},\ \{ e^{(2)}, e^{(5)}, e^{(11)} \},\ \{ e^{(2)}, e^{(7)}, e^{(9)} \}, \right. \\
    \left. \{ e^{(3)}, e^{(6)}, e^{(9)} \},\ \{ e^{(3)}, e^{(8)}, e^{(11)} \},\ \{ e^{(4)}, e^{(5)}, e^{(10)} \},\ \{ e^{(4)}, e^{(7)}, e^{(12)} \} \right\}.
\end{multline*}
Since
\begin{equation*}
\sign \det( e^{(1)} - e^{(6)} ) = \sign \det( e^{(6)} - e^{(12)} ) = \sign\det( e^{(12)} - e^{(1)} ) \neq 0
\end{equation*}
and likewise for the other 3-tuples, cf.\ Table~\ref{tab:compatibility}, we obtain by Lemma~\ref{lem:T3} that each of these 3-tuples forms a $T_3$. (See also Remark~\ref{rem:incompatible-distances}.)

Let $v_{i,j,k} = \{ e^{(i)}, e^{(j)}, e^{(k)} \} \in \E^3_{\incompatible}$. We set
\begin{align*}
\tau_{i,j,k} &= \RH(v_{i,j,k}), \\
\T &:= \{ \tau_{i,j,k}\ |\ v_{i,j,k} \in \E^3_{\incompatible} \}.
\end{align*}

For $\tau \in \T$ and $r \in S_4 \times C_2$ by $r \tau$ we mean the $T_3$ formed by the image under $r$ of the vertices of $\tau$. (The existence of such $T_3$s follows from $S_4 \times C_2$ being a symmetry group of $\E^3_{\incompatible}$ as was shown in Lemma~\ref{lem:symmetry2}). The symmetry relations between the $T_3$s is illustrated in Figure~\ref{fig:T3-symmety}. As we shall see (Example~\ref{eg:toblerone} below) each $\tau \in \T$ is specially related to $r_0 \tau$; we refer to it as the \emph{dual} of $\tau$.

\begin{figure}[ht]
\begin{center}
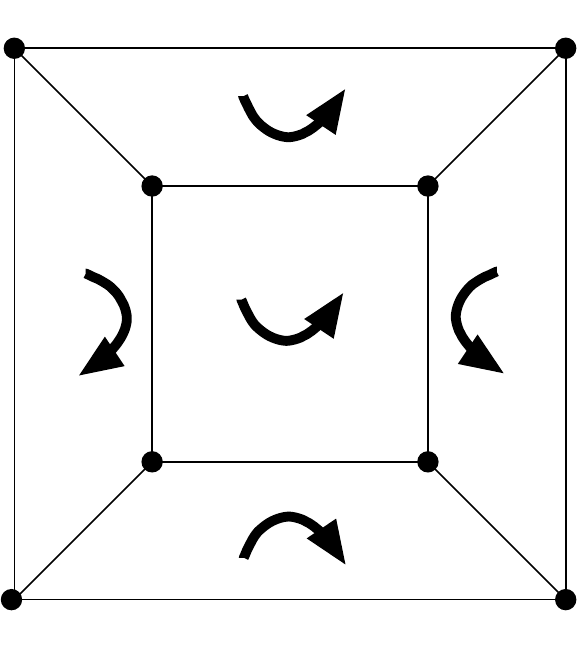 \quad \quad
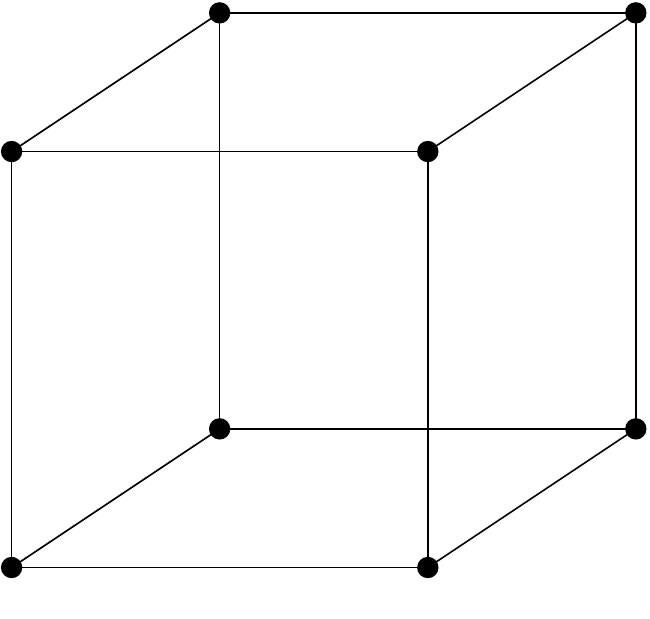
\caption{The symmetry of the eight $T_3$s. }
\label{fig:T3-symmety}
\end{center}
\end{figure}

Before we proceed further we note that each of these eight $T_3$s is symmetric (Definition~\ref{def:T3-symmetry}) and has distinct nodes (Definition~\ref{def:T3-nodes}). Moreover all eight $T_3$s are similar (Definition~\ref{def:T3-similarity}):

\begin{lemma}
\label{lem:T3-equivalence}
All $T_3$s in $\T$ are similar, and each is symmetric and has distinct nodes.
\end{lemma}

\begin{proof}
The symmetry of each $T_3$ and the similarity of all $T_3$s in $\T$ follow from Remark~\ref{rem:incompatible-distances}.

Now consider the nodes of $\tau_{1,8,10}$. Suppose, on the contrary, that $e_{1,8} = e_{8,10} = e_{10,1}$. (Here we use the notation of Section~\ref{sec:T3}.) Then, by symmetry, the nodes coincide with the barycentre of the $T_3$ which is $\frac{1}{3} \left( e^{(1)} + e^{(8)} + e^{(10)} \right)$. However, an elementary calculation shows that this barycentre is incompatible with $e_1$, $e_8$ and $e_{10}$ which is a contradiction (cf.\ Definition~\ref{def:T3-1}). Thus $\tau_{1,8,10}$ has distinct nodes and by symmetry this is true for all eight $T_3$s.
\end{proof}

From Proposition~\ref{prop:T3-hull} and symmetry it follows that the barycentre of each $v \in \E^3_{\incompatible}$ is contained in $\RH(v)$.

From Observations~\ref{obs:e<d} and~\ref{obs:e=d}, when $\eps \leqslant \delta$,  each $T_3$ in $\T$ along with its dual is contained in a facet of $\CH(\E)$. In particular for each $\tau \in \T$, the symmetrised rank-one convex hull of $\tau \cup r_0 \tau$ is contained in a facet of $\CH(\E)$ and is thus a part of the boundary of the symmetrised rank-one convex hull of $\E$:
\begin{equation} \label{eq:T3-pairs-are-on-boundary}
 \RH(\tau \cup r_0 \tau) \subset \partial \mathcal{R}(\E).
\end{equation}

Example~\ref{eg:toblerone} below reveals that $\RH(\tau \cup r_0 \tau)$ is four dimensional and that each point in it can be attained by laminates of $T_3$ microstructures. In other words $\RH(\tau \cup r_0 \tau)$ contains a four-dimensional set of $T_3$s. Before we show this we first construct a three-dimensional set of $T_3$s which lies in $\partial \mathcal{R}(\E)$ when $\eps \leqslant \delta$.

\begin{example}[A three-dimensional set of $T_3$s] \label{eg:3D-set-T3s}
Let $\tau \in \T$. From Table~\ref{tab:compatibility} the elements of $\E$ with which all the vertices of $\tau$ are compatible are precisely the vertices of $r_0 \tau$. Following the construction introduced in Lemma~\ref{lem:T3-3D}, we construct a three-dimensional set of $T_3$s from $\tau$ and (any) one vertex from $r_0 \tau$. There are three such continua of $T_3$s, one for each vertex of $r_o \tau$. Analogously there are three such sets constructed from $r_0 \tau$ and the vertices of $\tau$. 
\end{example}

Note that this example shows that $T_3$s exist not only in $\E$ but also in $\LH_1(\E) \setminus \E$ since the vertices of almost all of the $T_3$s so constructed are themselves attained by a lamination of strains in $\E$. However in the next example the construction of $T_3$s precedes the construction of laminates: 

\begin{example}[A four-dimensional set of laminates of nodes of $T_3$s] \label{eg:toblerone}
From each $\tau \in \T$ and its dual we construct a four-dimensional set of laminates of nodes of $T_3$s. We do this explicitly for $\tau_{1,8,10}$ and its dual $\tau_{2,7,9}$; the construction for the other pairs is similar.

First we note that each point in one of these $T_3$s is compatible with the corresponding point (i.e., the point with the same barycentric coordinates) in its dual. In fact this is true even after a cyclic permutation of the vertices: $\forall x,y,z \in \R$,
\begin{equation*}
 x e^{(1)} + y e^{(8)} + z e^{(10)} \, \compatible \, x e^{(2)} + y e^{(9)} + z e^{(7)}, \, x e^{(7)} + y e^{(2)} + z e^{(9)}, \, x e^{(9)} + y e^{(7)} + z e^{(2)}. 
\end{equation*}
This is immediate from the calculation~\cite[Pair\_of\_Level-1\_T3s.nb]{Chenchiah-Schloemerkemper1ESM}:
\begin{align*}
 &\det \left( ( x e^{(1)} + y e^{(8)} + z e^{(10)} ) - ( x e^{(2)} + y e^{(9)} + z e^{(7)} ) \right) \\
 &= \det \left( ( x e^{(1)} + y e^{(8)} + z e^{(10)} ) - ( x e^{(7)} + y e^{(2)} + z e^{(9)} ) \right) \\
 &= \det \left( ( x e^{(1)} + y e^{(8)} + z e^{(10)} ) - ( x e^{(9)} + y e^{(7)} + z e^{(2)} ) \right) \\
 &= 0.
\end{align*}
Thus, in particular, each node of $\tau_{1,8,10}$ is compatible with every node of $\tau_{2,7,9}$ (and vice versa). Since the nodes of a $T_3$ are pair-wise compatible it follows that these six nodes (i.e.\ in the notation of Section~\ref{sec:T3}, $e_{1,8}$, $e_{8,10}$, $e_{10,1}$, $e_{2,9}$, $e_{9,7}$ and $e_{7,2}$) are pair-wise compatible (see Figure~\ref{fig:toblerone}, note that the figure is schematic; in fact $\tau$ and $r_0 \tau$ share the same barycenter). We conclude that:
\begin{equation*}
\CH(\{ e_{1,8}, e_{8,10}, e_{10,1}, e_{2,9}, e_{9,7}, e_{7,2} \}) \subset \RH(\{ e_1, e_2, e_7, e_8, e_9, e_{10} \}).
\end{equation*}
Each point in this convex hull is attained by a lamination of the nodes of $\tau_{1,8,10}$ and $\tau_{2,7,9}$. A simple Mathematica verification~\cite[Dimension.nb]{Chenchiah-Schloemerkemper1ESM} shows that this convex hull is four-dimensional. (Thus the maximum depth of lamination required is also four, cf.\ proof of Theorem~\ref{thm:edge-compatible}.)
\end{example}

\begin{figure}
\begin{center}
\resizebox{13cm}{!}{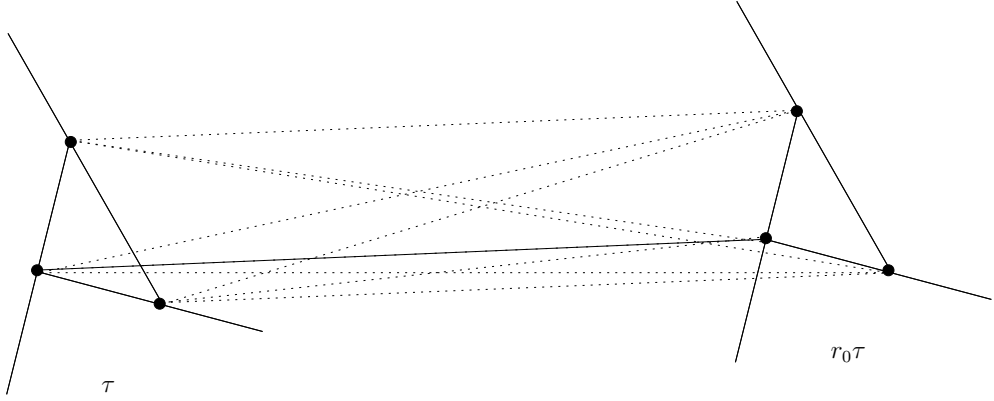}
\caption{Schematic representation of the four-dimensional set of laminates of nodes of $T_3$s constructed in Example~\ref{eg:toblerone}.}
\label{fig:toblerone}
\end{center}
\end{figure}

Note that when $\eps \leqslant \delta$, from~\eqref{eq:T3-pairs-are-on-boundary},
\begin{equation} \label{eq:toberone-is-on-boundary}
 \CH(\{ e_{1,8}, e_{8,10}, e_{10,1}, e_{2,9}, e_{9,7}, e_{7,2} \}) \subset \partial \mathcal{R}(\E).
\end{equation}

\subsection{Level-2 $T_3$s and related microstructures}
\label{sec:T3s-level2}

Next we construct new $T_3$s from the nodes of the $T_3$s in $\T$. We refer to the former $T_3$s as Level-1 $T_3$s and to the new $T_3$s as Level-2 $T_3$s. Level-2 $T_3$s allow us to construct a five-dimensional set of $T_3$s, see Construction~\ref{cons:5D-set-T3s} below.

\begin{cons} \label{cons:level2-T3s}
Let $\tau \in \T$ and let $\tau_1,\tau_2,\tau_3 \in \T$ be chosen such that, in Figure~\ref{fig:T3-symmety}, the line joining $\tau$ and $\tau_i$ is an edge of the cube for $i=1,2,3$ (thus $\tau_i = r_j^{\pm1} \tau$ for some $j=1,2,3$). Note that the set $\{\tau_1,\tau_2,\tau_3\}$ is invariant under any element of $S_4$ that leaves $\{\tau, r_0 \tau\}$ invariant (these are rotations of $\frac{2\pi}{3}$ through the major diagonal formed by $\tau$ and $r_0 \tau$). Let $r$ be an element of this group (i.e., one of two such rotations).
Now let $n_1$ be a node of $\tau_1$. Then $n_1$, $r n_1$, $r^2 n_1$ form a symmetric $T_3$ with distinct nodes. Similarly for $n_2$, $r n_2$, $r^2 n_2$ and $n_3$, $r n_3$, $r^2 n_3$. This is illustrated in Figure~\ref{fig:level2-T3s}. 
\end{cons}

\begin{figure}
\begin{center}
\resizebox{12.5cm}{!}{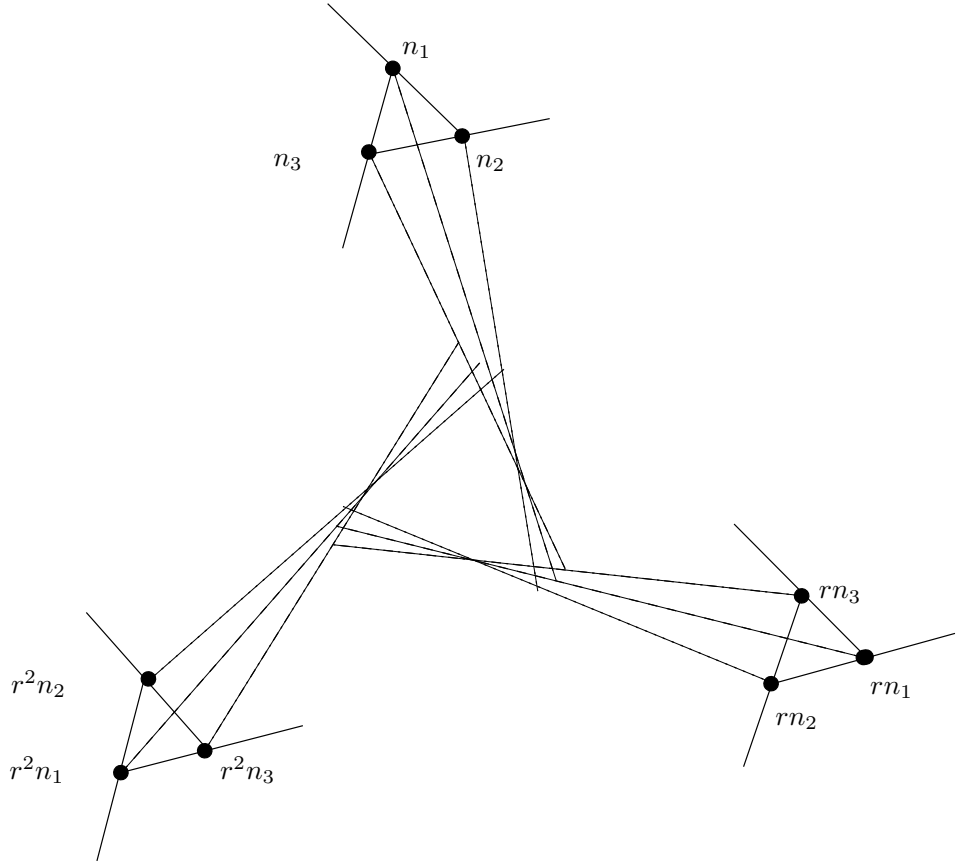}
\caption{Level-2 $T_3$s.} 
\label{fig:level2-T3s}
\end{center}
\end{figure}

\begin{proof}
We show that $n_1$, $r n_1$, $r^2 n_1$ form a $T_3$ explicitly for $\tau = \tau_{3,8,11}$, $\tau_1 = \tau_{1,8,10}$, $\tau_2 = \tau_{2,5,11}$ and $\tau_3 = \tau_{3,6,9}$; the result for the other 3-tuples follows by symmetry.

Since $n_1 \in \tau_{1,8,10}$, it has the barycentric representation $x e^{(1)} + y e^{(8)} + z e^{(10)}$ for some $x,y,z \in [0,1]$ with $x+y+z=1$. Then, as can be easily checked (Table~\ref{tab:generators}),
\begin{equation*}
\{r n_1, r^2 n_1 \} = \{x e^{(5)} + y e^{(11)} + z e^{(2)}, x e^{(9)} + y e^{(3)} + z e^{(6)} \}.
\end{equation*}
(Note the order of the vertices.) It can be verified~\cite[Level-2\_T3s.nb]{Chenchiah-Schloemerkemper1ESM} that
\begin{align*}
 &\det \left( ( x e^{(1)} + y e^{(8)} + z e^{(10)} ) - ( x e^{(5)} + y e^{(11)} + z e^{(2)} ) \right) \\
 &= \det \left( ( x e^{(5)} + y e^{(11)} + z e^{(2)} ) - ( x e^{(9)} + y e^{(3)} + z e^{(6)} ) \right) \\
 &= \det \left( ( x e^{(9)} + y e^{(3)} + z e^{(6)} ) - ( x e^{(1)} + y e^{(8)} + z e^{(10)} ) \right) \\
 &\neq 0.
\end{align*}
Thus, by Lemma~\ref{lem:T3}, Remark~\ref{rem:T3-symmetry} and~\eqref{eq:coordinate-rotations}, $n_1$, $r n_1$, $r^2 n_1$ form a symmetric $T_3$. (From Remark~\ref{rem:T3-symmetry} it would have sufficed to check that one of the determinants above is non-zero.) To show that the nodes of this $T_3$ are distinct it suffices to check that the barycentre of the $T_3$ is incompatible with (one of) its nodes (cf.\ proof of Lemma~\ref{lem:T3-equivalence}). From~\cite[Level-2\_T3s.nb]{Chenchiah-Schloemerkemper1ESM}:
\begin{multline*}
x e^{(1)} + y e^{(8)} + z e^{(10)},\ x e^{(5)} + y e^{(11)} + z e^{(2)},\ x e^{(9)} + y e^{(3)} + z e^{(6)} \\
 \incompatible \frac{1}{3} \left( ( x e^{(1)} + y e^{(8)} + z e^{(10)} ) + ( x e^{(5)} + y e^{(11)} + z e^{(2)} ) + ( x e^{(9)} + y e^{(3)} + z e^{(6)} ) \right).
\end{multline*}
Which completes the proof.
\end{proof}

Since there are eight choices of $\tau$ and three choices of $n_1$ for each choice of $\tau$, by this construction we obtain 24 $T_3$s. 

Finally Example~\ref{eg:toblerone} and Construction~\ref{cons:level2-T3s} can be combined to construct a five-dimensional set of $T_3$s whose vertices are themselves laminates of nodes of (level-1) $T_3$s.

\begin{cons} \label{cons:5D-set-T3s}
Let $\tau \in \T$ and let $\tau_1, \tau_2, \tau_3 \in \T$ and $r \in S_4$ be as in Construction~\ref{cons:level2-T3s} above. Let $n_{1,i}$ and $n'_{1,i}$, $i=1,2,3$ be the nodes of $\tau_1$ and its dual respectively. Pick $\mu_i, \mu'_i \in [0,1]$ such that $\sum_{i=1}^3 \mu_i + \sum_{i=1}^3 \mu'_i = 1$ and let
\begin{equation*}
p_1 = \sum_{i=1}^3 \mu_i n_{1,i}+ \sum_{i=1}^3 \mu'_i n'_{1,i}.
\end{equation*} 
Note that $p_1$ is an element of the four-dimensional set constructed in Example~\ref{eg:toblerone} and thus can be attained by a lamination of $n_{1,i}, n'_{1,i}$, $i=1,2,3$. We assert that $p_1$, $r p_1$, $r^2 p_1$ form a symmetric $T_3$. The union of the $T_3$s  as $p_1$ varies yields a five-dimensional set. When $\eps \leqslant \delta$ from~\eqref{eq:toberone-is-on-boundary},  this set intersects the boundary of the symmetrised rank-one convex hull of $\E$.
\end{cons}

\begin{proof}
 A Mathematica calculation~\cite[Level-2\_T3s.nb]{Chenchiah-Schloemerkemper1ESM} shows this explicitly for $\tau_1 = \tau_{1,8,10}$, $\tau_2 = \tau_{2,5,11}$ and $\tau_3 = \tau_{3,6,9}$ (i.e., $\tau = \tau_{3,8,11}$), the construction for the other 3-tuples is similar.

To see that the set of $T_3$s constructed here is five-dimensional, it suffices to note that $p_1$ is picked from a four-dimensional set which, not being closed under $r$, does not contain the $T_3$s formed by $p_1$, $r p_1$, $r^2 p_1$. It follows that the union of the $T_3$s (as $p_1$ varies) constructed here is a five-dimensional set. 
\end{proof}

It is natural at this point to ask whether the nodes of level-2 $T_3$s form level-3 $T_3$ and, more generally, whether the nodes of level-$n$ $T_3$s from level-$(n+1)$ $T_3$s. We postpone these questions to~\cite{Chenchiah-Schloemerkemper-JMPS} and instead conclude by considering some implications of the results presented here.

\section{Conclusions}
\label{sec:conclusions}

\subsection{Mathematical comments}

While we have, in the later half of this paper, focused on monoclinic-I martensite, it is clear that our general strategy can, in principle, be applied to any finite set in $\Sym[c]{3}$; indeed in~\cite{Chenchiah-Schloemerkemper-JMPS} we apply it also to monoclinic-II martensite. Here we briefly comment on the two main components of our strategy, namely an understanding of the algebraic structure of symmetrised rank-one convex cones, and an understanding of the polytope structure of the given set.

\paragraph{The algebraic structure of symmetrised rank-one convex cones.}
While we have a complete understanding of symmetrised rank-one convexity in two-dimensions (Sections~\ref{sec:2D-cone} and~\ref{sec:T3}, and \cite{Chenchiah-Schloemerkemper-PRSL}), the algebraic structure of symmetrised rank-one convex cones is not yet sufficiently well understood in higher dimensions. A key missing ingredient is a characterisation, in terms of canonical forms, of real cubic polynomials in several variables. This problem in invariant theory seems unsolved for three and more variables. The fruitfulness of our approach in two-dimensions suggests that it might be valuable to more fully explore the algebraic aspects of symmetrised rank-one convexity.

When $\eps < \delta$ we demonstrate the existence of $T_3$s that attain points on the boundary of $\CH(\E)$, more precisely, that attain points on the four-dimensional facets of $\CH(\E)$. Though (as a Mathematica calculation~\cite[Facet-T3Pair.nb]{Chenchiah-Schloemerkemper1ESM} shows the $T_3$s do not belong to the three-dimensional facets (of the four-dimensional facets) of $\CH(\E)$, we suspect that the symmetrised rank-one convex hull of monoclinic-I martensite is strictly larger than the lamination hull. If so, the question arises as to how much larger it is. In terms of dimensions a perturbation argument shows that $\RH(\E) \setminus \LH(\E)$ would be at least two-dimensional. In fact we suspect that it is full (i.e., five) dimensional.

\paragraph{Convex polytopes.}
Since the convex hull of any finite set is a convex polytope, it is natural that an attempt to determine the semi-convex hull of a finite sets takes advantage of the structure of the convex polytopes they generate. This seems not to have been considered in the literature except for Theorem~\ref{thm:edge-compatible}. Lemma~\ref{lem:edges} and Remark~\ref{rem:local1} demonstrate the counter-intuitive behaviour of high-dimensional polytopes and thus the usefulness of knowledge of the theory of convex polytopes.

Moreover, when, as in the example of monoclinc-I martensite considered here, the faceting structure of the polytope depends on the material parameters, it might be expected that qualitative features of the semi-convex hulls and envelopes depend on the material parameters as well. If so, this heightens the possible utility of these polytopes for evaluating semi-convex hulls and envelopes.

\subsection{Implications for mechanics}

\paragraph{Two kinds of monoclinic-I martensite.}
Curiously, all cubic-to-monoclinic-I materials that we are aware of are monoclinic-Ia martensites (i.e., those for which $\eps<\delta$), see Table~\ref{tab:parameters}. It is natural to ask whether monoclinic-I martensite recovers more strains (modulo appropriate normalisation of the lattice parameters) as $\eps-\delta$ approaches zero (with $\eps=\delta$ being the ideal), and whether monoclinic-Ib martensites (i.e., those for which $\eps>\delta$) would demonstrate greater shape memory effect.

\paragraph{Monoclinic-II martensite.}
We have reason to believe that there are multiple kinds of monoclinic-II martensite as well. We hope to settle this question in~\cite{Chenchiah-Schloemerkemper-JMPS}.

As can be easily verified (eg.\ \cite{Bhattacharya:2003}), the compatibility relations between the twelve transformation strains of monoclinic-II martensite are identical to those between the twelve transformation strains of monoclinic-I martensite. Then, with the help of Lemma~\ref{lem:T3}, exactly as for monoclinic-I martensite, eight $T_3$s can be formed from these strains. This positively answers the question raised in \cite[p863]{Bhattacharya:1994p843} as to whether $T_3$s can be formed from the twelve transformation strains of monoclinic-II martensite. Indeed Lemma~\ref{lem:T3} presents an elementary test by which this question can be answered for any three-tuple of strains that have the same trace.

\paragraph{Special parameters.}
Throughout this paper, including in the (non-numerical) Mathematica computations, we have assumed the lattice parameters to be generic except that we considered the case $\eps=\delta$. We did identify one special case, $(\alpha-\beta) \delta + \eps^2 - \delta^2 = 0$, in which all twelve strains of monoclinic-I martensite are pair-wise compatible.

The question arises as to whether a material recovers more strains as $(\alpha-\beta) \delta + \eps^2 - \delta^2$ approaches zero. We suggest another parameter of importance, $\lambda := \lambda_{12} = \lambda_{23} = \lambda_{31}$ (Definition~\ref{def:T3-symmetry}) which however appears related to $(\alpha-\beta) \delta + \eps^2 - \delta^2$ for the three materials considered in Table~\ref{tab:special-parameters}. 

\begin{table}
\begin{center}
\begin{tabular}{c|c|c}
Material & $\lambda$ & $(\alpha-\beta) \delta + \eps^2 - \delta^2$ \\ \hline
NiTi & $0.6830$ & $0.0024$ \\ \hline
TiNiCu & $0.6683$ & $0.0021$ \\ \hline
CuZr & $0.0396$ & $-0.0015$ \\
\end{tabular}
\caption{$\lambda$ and $(\alpha-\beta) \delta + \eps^2 - \delta^2$ for the $T_3$s in $\T$ for NiTi, CuZr and TiNiCu.}
\label{tab:special-parameters}
\end{center}
\end{table}

Our reasoning is that as $\lambda$ becomes close to either $0$ or $1$, the nodes of a $T_3$ become closer to its vertices and the energetic penalty for a $T_3$ microstructure being approximated by a finite-rank laminate becomes smaller. Indeed,  instead of constructing a finite-rank laminate from the three nodes of a $T_3$, it suffices to move only  \emph{one} of the vertices to a node.  

As Table~\ref{tab:special-parameters} shows $\lambda = 0.0396$ for CuZr. We hypothesise that for this material the symmetrised lamination convex hull is very close to the convex hull and thus the convex hull is a very good approximation to all its semi-convex hulls. The same would apply to the (semi-)convex envelopes of the corresponding energy density as well. A similar reasoning might explain the remarkable closeness between the symmetrised lamination convex hull and the convex hull for CuAlNi (a monoclinic-II martensite) observed in \cite{Govindjee:2007}.

\paragraph{Microstructure corresponding to $T_3$s.}
For monoclinic-I martensite with $\eps < \delta$ we have shown that there exists a five-dimensional set of $T_3$s which reaches the boundary of $\CH(\E)$. This raises the question as to whether the microstructures experimentally observed for strains in this set are laminate approximations to $T_3$ microstructures. We wonder too if experimental observations of such microstructures would provide insight into mechanisms governing microstructure formation (dynamics) and the role of surface energy. 

\paragraph{Acknowlegements.}
This research was supported by the Royal Society (International Joint Project ``Symmetrised rank-one convex hull of monoclinic martensite''). It was initiated while the authors were at the Max Planck Institute for Mathematics in the Sciences, Leipzig and continued while AS was at the Department of Mathematics, University of Erlangen-Nuremberg, the Hausdorff Center for Mathematics and the Institute for Applied Mathematics, University of Bonn. IVC thanks these institutes and the Institute for Mathematics, University of W\"urzburg for hospitality during visits. IVC was introduced to the problem of determining the semi-convex hulls of monoclinic-I martensite when Sanjay Govindjee and Valery Smyshlyaev visited Kaushik Bhattacharya at Caltech. We thank Kaushik Bhattacharya, Richard James, Robert Kohn and Stefan M\"uller for helpful discussions. We thank the referees for helpful comments, for suggesting the name ``wave convex'', and for bringing \cite{Tartar:1979,Murat:1981} to our attention.

\bibliographystyle{spmpsci}
\bibliography{References,Papers,Manuscripts}

\end{document}